\documentclass[11pt,a4paper]{article}       % onecolumn (second format)

\usepackage{amssymb}
\usepackage{graphicx,wrapfig,lipsum}
\usepackage{amsmath}
\usepackage{tabularx}
\usepackage{epstopdf}
\usepackage{amssymb}
\usepackage{caption}
\usepackage{authblk}
\usepackage{xcolor}
\usepackage{subcaption}
\usepackage{float}

\usepackage{enumerate}

\newcommand{\norm}[1]{\left\lVert#1\right\rVert}

\newcommand{\uw}{{\bf{u_w}}}
\newcommand{\uwn}{\bf{u^n_w}}
\newcommand{\uwo}{{\bf{u^{n-1}_w}}}
\newcommand{\ez}{{\bf{e_z}}}

\newtheorem{remark}{Remark}
\newtheorem{theorem}{Theorem}
\newtheorem{proof}{Proof}
\newtheorem{acknowledgments}{Acknowledgments}

\usepackage{mathptmx}      % use Times fonts if available on your TeX system

\begin{document}

\title{Iterative schemes for surfactant transport in porous media\thanks{Project founded by VISTA, a collaboration between the Norwegian Academy of Science and Letters and Equinor.}
}
\author[1]{Davide Illiano\thanks{davide.illiano@uib.no}}
\author[1,2]{Iuliu Sorin Pop\thanks{sorin.pop@uhasselt.be}}
\author[1]{Florin Adrian Radu\thanks{florin.radu@uib.no@university.edu}}
\affil[1]{Department of Mathematics, University of Bergen}
\affil[2]{Faculty of Sciences, University of Hasselt}

\renewcommand\Authands{ and }
%\institute{Davide Illiano, Florin Adrian Radu \at
%              Department of Mathematics, University of Bergen, Allegaten 41, Bergen, Norway \\
              %Tel.: +47 46522618\\
%              \email{\{Davide.Illiano, Florin.Radu\}@uib.no}  
%           Iuliu Sorin Pop\at
%           Faculty of Sciences, University of Hasselt, Agoralaan Building D, BE 3590 Diepenbeek, Belgium\\
           %Tel.: \\
%          \email{sorin.pop@uhasselt.be}  
%}

\maketitle
\textbf{Abstract} In this work we consider the transport of a surfactant in variably saturated porous media. The water flow is modelled by the Richards equations and it is fully coupled with the transport equation for the surfactant. Three linearization techniques are discussed: the Newton method, the modified Picard and the L-scheme. Based on these, monolithic and splitting schemes are proposed and their convergence is analyzed. The performance of these schemes is illustrated on four numerical examples. For these examples, the number of iterations and the condition numbers of the linear systems emerging in each iteration are presented.



\section{Introduction}
\label{intro}
\paragraph{}
Many societal relevant problems are involving multiphase flow and multicomponent, reactive transport in porous media. Examples in this sense appear in the enhanced oil recovery, geological $CO_2$ storage, diffusion of medical agents into the human body, or water or soil pollution. \textcolor{blue}{In all these situations, experimental results are difficult and expensive to obtain, therefore numerical simulations become a key technology. Together with lab experiments and field data, they can help us comprehend these complex phenomena.} The mathematical models for problems as mentioned above are (fully or partially) coupled, non-linear, possible degenerate partial differential equations. In most cases, deriving explicit solutions is not possible, whereas developing appropriate algorithms for finding numerical solutions is a challenge in itself. Here we investigate robust and efficient methods for solving the nonlinear problems obtained after performing an implicit time discretization, the focus being on iterative, splitting or monolithic schemes for fully coupled flow and transport.
\par
%surfactant transport
Of particular interest here is a special case of multiphase, reactive flow in porous media, namely the surfactant transport in soil \cite{agosti,surfactant1999,prechtel,Karagunduz,surf,kumar2013}. Surfactants, which are usually organic compounds, are commonly used for actively combating soil and water pollution \cite{Gallo,raduAWR,christofi,suciu2014,dawson1998}. They contain both hydrophobic and hydrophilic groups and are dissolved in the water phase, being transported by diffusion and convection. Typically, the surfactants are employed in soil regions near the surface (vadose zone), where water and air are present in the pores. Consequently, the outcoming mathematical model accounts the transport of at least one species (the surfactant, but often also the contaminant) in a variably saturated porous medium. Whereas the dependence of the species transport on the flow is obvious, one can encounter the reverse dependence as well when surfactants are affecting the interfacial tension between water and air, leading to a dependency of the water flow on the concentration of surfactant. In other words, one has to cope with a fully coupled flow and transport problem, and not only with a one-way coupling, i.e. when only the transport depends on the flow, as mostly considered in reactive transport \cite{radu2010}. 
\par
%Mathematical model Richards equation
Whereas the surfactant transport is described by a reaction-diffusion-convection equation, water flow in variably saturated porous media is modelled by the Richards equation \cite{RichardsBerardi,BookHelmig}.  The main assumption in this case is that the air remains in contact with the atmosphere, having a constant pressure (the atmospheric pressure, here assumed zero). This allows reducing the flow model to one equation, the Richards' equation. In mathematical terms, this equation is degenerate parabolic, whose solution has typically low regularity \cite{AltLuckhaus}. 

From the above, and adopting the pressure head as the main unknown in the Richards' equation, we study here different linearization schemes for the model 
%\par
%Find $\Psi(t, x, y, z)$ and $c(t,  x, y, z) $ such that there holds in $(0, T) \times \Omega$
%%For discretization techniques, convergence analysis and linearization schemes we refer to e.g \cite{convergence,pop,putti,radu2014,pop2002,eymard1999,schneid2004,radu2004}.
\begin{equation}\label{richardscoupled}
\frac{\partial \theta(\Psi, c)}{\partial t} - \nabla \cdot (K(\theta(\Psi, c))\nabla(\Psi+z))\ =\ H_1 
\end{equation}
and 
\begin{equation}\label{transpcoupled}
\frac{\partial \theta(\Psi, c) c}{\partial t} - \nabla \cdot (D\nabla c - \uw c) + R(c)\ =\ H_2,
\end{equation}
holding for $\vec{x} \in \Omega$ ($z$ being the vertical coordinate of $\vec{x}$, pointing against gravity) and $t \in (0, T]$. Here $\Omega$ is a bounded, open domain in $\mathbb{R}^d$ ($d = 1, 2$ or $3$) having a Lipschitz continuous boundary $\partial \Omega$ and $T > 0$ is the final time.  Further, $\theta(\cdot, \cdot)$ denotes the water content, and is a given function depending on the pressure head $\Psi$ and of the surfactant concentration $c$. Also, $K(\cdot)$ is the hydraulic conductivity, $D>0$ the diffusion/dispersion coefficient. Finally, ${\uw}:= - K(\theta,c)\nabla(\Psi+z)$ is the water flux, $R(\cdot)$ the reaction term expressed as a function of the concentration $c$, and $H_1, H_2$ are the external sinks/sources. Initial and boundary conditions, which are specified below, complete the system.
\par
We point out that the water content and the hydraulic conductivity,  $\theta(\cdot, \cdot)$ and $K(\cdot)$ are given non-linear functions. They are medium- and surfactant-dependent and are determined experimentally (see \cite{BookHelmig}). Specific choices are provided in Section \ref{numericalrichards}. 
\par
%discretization in time and space; linearization. 
To solve numerically the system (\ref{richardscoupled}) -- (\ref{transpcoupled}) one needs to discretize in time and space. We refer to \cite{Farthing} for a practical review of numerical methods for the Richards equation. Due to the low regularity of the solution and the need of relatively large time steps, the backward Euler method is the best candidate for the time discretization. Multiple spatial discretization techniques are available, such as the Galerkin Finite Element Method (\emph{FEM}) \cite{Nochetto,barrett1997,russell1983}, the Mixed Finite Element Method (\emph{MFEM}) \cite{arbogast1996,radu2004,vohralik,woodward}, the Multi-Point Flux Approximation (\emph{MPFA}) \cite{klausenMPFA,bause2010,aavatsmark} and the Finite Volume Method (\emph{FVM}) \cite{cances,eymard1999,eymard2006}. 

Since the time discretization is not explicit, the outcome is a sequence of non-linear problems, for which a linearization step has to be performed. Widely used linearization schemes are the quadratic, locally convergent Newton method and the modified Picard method \cite{celia}. For both, the convergence is guaranteed if the starting point is close to the solution. Since for evolution equations the initial guess is typically the solution at the previous time, this may induces severe restrictions on the time step size (see \cite{pop2}). Among alternative approaches we mention the L-scheme (see \cite{Pop98,Slodicka,Pop2004,List2016}) and the modified L-scheme \cite{Mitra}, both being robust w.r.t. the mesh size, but converging linearly. In particular, the L-scheme converges for any starting point, and the restriction on the time step, if any, is very mild. The modified L-scheme makes explicit use of the choice of the starting point as the solution obtained at the previous time, and has an improved convergence behaviour if the changes in the solutions at two successive times are controlled by the time step. Neverhteless, the modified L-scheme involves computation of derivatives while the L-scheme does not. Finally, the robustness of the Newton method is significantly increased if one considers combinations of the Picard and the Newton methods \cite{putti}, and in particular of the L-scheme and the Newton scheme \cite{List2016}. 
\par
We conclude this discussion by mentioning that in this paper we adopt the \emph{FEM} and the \emph{MPFA}, but the iterative schemes presented here can be applied in combination with any other spatial discretization. The focus is on effectively solving the flow and transport system (\ref{richardscoupled}) -- (\ref{transpcoupled}), and in particular on the adequate treating of the coupling between the two model components (the flow and the reactive transport). The schemes are divided in three main categories: monolithic (Mon), non-linear splitting (NonLinS) and alternate splitting (AltS). Subsequently, we denote e.g. by Mon-NE, the monolithic scheme obtained by applying the Newton method as linearization. The nonlinear splitting schemes (NonLinS) should be understood as solving at each time step first the flow equation until convergence, by using the surfactant concentration from the last iteration, and then with the obtained flow solving the transport equation until convergence. The procedure can be continued iteratively, this being the usual {\it or classical} splitting method for transport problems. The convergence of NonLinS does not depend on the linearization approach used for each model component (Newton, Picard or L-scheme), because we assume that the nonlinear subproblems are solved exactly, i.e. until convergence. Finally, the alternate splitting methods (AltS) have a different philosophy. Instead solving each subproblem until convergence within each iteration, one performs only one step of the chosen linearization. For example, AltS-NE will perform one Newton step for each model component, and iterate. These schemes are illustrated in Figures \ref{scheme2}, \ref{scheme3}. 

All the schemes can be analysed theoretically, and we do this exemplary for Mon-LS, i.e. for the monolithic approach combined with the L-scheme. Based on comparative numerical tests performed for academic and benchmark problems, we see that the alternate methods can save substantial computational time, while maintaining the robustness of the L-scheme.
\par
The remaining of the paper is organized as follows. In Section \ref{numericalrichards} we establish the mathematical model and the notation used and present the iterative monolithic and splitting schemes.  In Section \ref{sec:convergence} we prove the convergence of the $Mon-LS$ scheme and briefly discuss the convergence of the other schemes. Section \ref{numericalexample} presents four different numerical examples. They are inspired by the cases already studied in the literature \cite{List2016,surf}. 
Section \ref{conclusion} concludes this work.

\section{Problem formulation, discretization and iterative schemes}\label{numericalrichards}
\paragraph{}
We solve the fully coupled system (\ref{richardscoupled})--(\ref{transpcoupled}), completed by homogeneous Dirichlet boundary conditions for both $\Psi$ and $c$ and the initial conditions: 
$$\Psi = \Psi_0 \text{ and } c = c_0  \text{ at } t = 0.$$ 
We use the van Genuchten-Mualem parameterization \cite{vanG}
\begin{equation} \label{theta}
\theta(\Psi) = \begin{cases} \theta_r + (\theta_s - \theta_r) \left( \frac{1}{1+(-\alpha \Psi)^n} \right)^{\frac{n-1}{n}}, &\Psi \leq 0 \\
\theta_s, &\Psi>0,
\end{cases}
\end{equation}
\begin{equation} \label{K}
K(\theta(\Psi)) = \begin{cases} K_s \theta_e^{\frac{1}{2}} \left[1-\left(1-\theta_e^{\frac{n}{n-1}}\right)^{\frac{n-1}{n}} \right]^2, &\Psi \leq 0 \\
K_s, &\Psi>0,
\end{cases}
\end{equation}
where $\theta_r$ and $\theta_s$ represent the values of the residual and saturated water content, $ \theta_e = (\theta(Psi)-\theta_r)/(\theta_s-\theta_r)$ is the effective water content, $K_s$ is the conductivity and $\alpha$ and $n$ are model parameters depending on the soil. 
\par
Observe that in the expression above for $\theta$ the influence of the surfactant on the water flow is neglected. As reported  in \cite{salt,surf,laboratory},  the surface tension between water and air does depend on the surfactant concentration $c$, implying the same for the function $\theta$ above. The following parametrization is proposed in \cite{surf}
\begin{equation}\label{rescaled}
\theta(\Psi, c) := \theta\Big(\frac{\gamma(c)}{\gamma_0(c_0)}\Psi \Big), \qquad \text{ with } \qquad \frac{\gamma(c)}{\gamma_0(c_0)} = \frac{1}{1- b \log (c/a +1)}. 
\end{equation}
Here $\gamma$ and $\gamma_0$ are the surface tensions at concentrations $c$ and $c_0$, the second being a reference concentration. The parameters $a$ and $b$ depend on the fluid and the medium. We refer to \cite{smith1994,laboratory} for details about the scaling factor in (\ref{rescaled}).
\par
This gives the following expressions for $\theta$ and $K$
\begin{equation} \label{thetarescaled}
\theta(\Psi, c) = \begin{cases} \theta_r + (\theta_s - \theta_r) \left[ 1/ \Big(1+\big(-\alpha (\frac{1}{1- b \log (c/a +1)}) \Psi\big)^n\Big) \right]^{\frac{n-1}{n}}, &\Psi \leq 0 \\
\theta_s, &\Psi>0,
\end{cases}
\end{equation}
\begin{equation} \label{Krescaled}
K(\theta(\Psi,  c)) = \begin{cases} K_s \theta_e^{\frac{1}{2}} \left[1-\left(1-\theta_e^{\frac{n}{n-1}}\right)^{\frac{n-1}{n}} \right]^2, &\Psi \leq 0 \\
K_s, &\Psi>0.
\end{cases}
\end{equation}
This shows that the flow component also depends on the reactive transport, implying that the model is coupled in both directions.
\par
In the following we proceed by discretizing the equations (\ref{richardscoupled}) and (\ref{transpcoupled}) in time and space. We will use common notations in functional analysis. We denote by $L^2(\Omega)$ the space of real valued, squared integrable function defined on $\Omega$ and $H^1(\Omega)$ its subspace, containing the functions having also the first order derivatives in $L^2(\Omega)$. $H_0^1(\Omega)$ is the space of functions belonging to $H^1(\Omega)$ and vanishing on $\partial \Omega$. Further, we denote by $< \cdot, \cdot > $ the $L^2(\Omega)$ scalar product (and by $\norm{\cdot}$ the associated norm) or the pairing between $H1_0$ and its dual $H^{-1}$. Finally, by $L^2(0, T; X)$ we mean the Bochner space of functions taking values in the Banach-space $X$,  the extension to $H^1(0, T; X)$ being straightforward. 
\par
With this we state the weak formulation of the problem related to (\ref{richardscoupled}) -- (\ref{transpcoupled}):
\par
\textbf{Problem P}: Find $\Psi, c \in L^2(0, T; H_0^{1}(\Omega)) \cap H^1(0, T; H^{-1}(\Omega))$ such that
\begin{equation}\label{ric1}
< \partial_t \theta(\Psi, c), v_1 > +  <K(\theta(\Phi, c))\nabla (\Psi + z) , \nabla v_1> = <H_1, v_1>
\end{equation}
and 
\begin{align}\label{tra1}
< \partial_t (\theta(\Psi,  c) c), v_2 > +  <D\nabla c + \uw c, \nabla v_2> = <H_2, v_2>
\end{align}
hold for all $v_1,v_2 \in H_0^{1}(\Omega)$ and almost every $t \in (0, T]$. 
\par
We now combine the backward Euler method with linear Galerkin finite elements for the  discretization of Problem P. We let $N \in \mathbb{N}$ be a strictly positive natural number and the time step $\tau \ := \ T/N$. Correspondingly, the discrete times are $t_n\ :=\ n\tau\ (n\in\{0, 1, \dots, N\})$. Further, we let  $T_h$ be a regular decomposition of $\Omega$, $\overline{\Omega} = \underset{T \in T_h}{\cup} T$ into $d$-dimensional simplices, with $h$ denoting the mesh diameter. The finite element space $V_h \subset H_0^1(\Omega)$ is defined by
\begin{equation}
V_h := \{ v_h\in H_0^1(\Omega)\ s.t.\ v_{h|T} \in \mathbb{P}_1(T), \text{ for any } T\in T_h\},
\end{equation}
where $\mathbb{P}_1(T)$ denotes the space of linear polynomials on $T$ and $v_{h|T}$ the restriction of $v_{h}$ to $T$.
\par
For the fully discrete counterpart of Problem P we let $n \ge 1$ be fixed and assume that $\Psi^{n-1}_h, c^{n-1}_h \in V_h$ are given. The solution pair at time $t_n$ solves 
\par
\textbf{Problem P$_n$}: Find $\Psi^{n}_h, c^{n}_h \in V_h$ such that for all $v_h, w_h \in V_h$ there holds 
\begin{equation}\label{ric2}
\begin{split}
&<\theta(\Psi^{n}_h, c^{n}_h) - \theta(\Psi^{n-1}_h, c^{n-1}_h), v_h > \\+ \tau <K(&\theta(\Psi^{n}_h, c^{n}_h))
(\nabla(\Psi^{n}_h)+ {\ez}), \nabla v_h > 
= \tau <H_1, v_h >
\end{split}
\end{equation}
and 
\begin{equation}\label{tra2}
\begin{split}
&< \theta(\Psi^{n}_h, c^{n}_h) c^{n}_h   - \theta(\Psi^{n-1}_h, c^{n-1}_h) c^{n-1}_h, w_h> \\
+& \tau  <D \nabla c^n_h + \uwo c^{n}_h, \nabla w_h>  = \tau <H_2, w_h>.
\end{split}
\end{equation}
%- \theta(\Psi^{n-1}_h, c^{n-1}_h) c^{n-1}_h), w_h + \tau  <D \nabla \Psi_h + \uwo c^{n}_h, \nabla w_h>  = \tau <H_2, w_h>
${\ez}$ denotes the unit vector in the direction opposite to gravity.
\par
\begin{remark} Observe that $\uwo$ appears in the convective term in \eqref{tra2}. This choice is made for the ease of presentation. Nevertheless, all calculations carried out in this paper were doubled by ones where $\uwn$ has replaced $\uwo$. The differences in the results were marginal.
\end{remark}
\par
Observe that Problem P$_n$ is a coupling system of two elliptic, nonlinear equations. In the following we discuss different iterative schemes for solving this system.

\subsection{Iterative linearization schemes}
\paragraph{}
We discuss monolithic and splitting approaches for solving Problem P$_n$, combined with either the Newton-method, or the modified Picard \cite{celia} or the L-scheme \cite{Pop2004,List2016}. In the following the index $n$ always refers to the time step, whereas $j$ denotes the iteration index. As a rule, the iterations start with the solution at the last time, $t_{n-1}$.

In the monolithic approach one solves the two equations of the system  \eqref{ric2}-\eqref{tra2} at once, and combined with a linearization method. Formally, this becomes 
\par
\textbf{Problem PMon$_{n, j+1}$}: Find  $\Psi^{n, j+1}$ and $c^{n, j+1}$ such that
\begin{equation}
\begin{cases} F^{lin}_1 (\Psi^{n, j+1},  c^{n, j+1}) = 0,\\
 F^{lin}_2 (\Psi^{n, j+1}, c^{n, j+1}) = 0.
\end{cases}
\end{equation} 
$F_k^{Lin}$ is a linearization of the expression $F_k$ ($k = 1, 2$) appearing in the system \eqref{ric2}-\eqref{tra2}. Depending on the used linearization technique, one speaks about a monolithic-Newton scheme (Mon-Newton), or monolithic-Picard (Mon-Picard) or monolithic-L-scheme (Mon-LS). These three schemes will be presented in detail below.

In the iterative splitting approach one solves each equation separately and then iterates between these using the results obtained previously. We distinguish between two main splitting ways: the nonlinear slitting and the alternate splitting. This is schematized in Figure \ref{scheme2} and Figure \ref{scheme3} respectively. The former becomes 
\par 
\textbf{Problem PNonLinS$_{n, j+1}$}: 
Find  $\Psi^{n, j+1}$ and $c^{n, j+1}$ such that
\begin{equation}
\begin{cases} F_1 (\Psi^{n, j+1}, c^{n,j}) = 0, \text{ followed by }\\
 F_2 (\Psi^{n, j+1}, c^{n, j+1}) = 0.
\end{cases}
\end{equation}
For the linearization of $F_1$ and $F_2$ one can use one of the three linearization techniques mentioned before. In contrast, in the alternate splitting one performs only one linearization step per iteration, see also Figure \ref{scheme3}. The alternate splitting scheme becomes 
\par
\textbf{Problem PAltS$_{n, j+1}$}: Find  $\Psi^{n, j+1}$ and $c^{n, j+1}$ such that
\begin{equation}
\begin{cases} F^{lin}_1 (\Psi^{n, j+1}, c^{n, j}) = 0, \text{ followed by } \\
 F^{lin}_2 (\Psi^{n, j+1}, c^{n, j+1}) = 0.
\end{cases}
\end{equation}
Depending on which linearization is used, one speaks about alternate splitting Newton (AltS-NE) or alternate splitting L-scheme (AltS-LS). Both schemes are presented in detail below.

\begin{figure}[H]
\centering
\includegraphics[scale=.4]{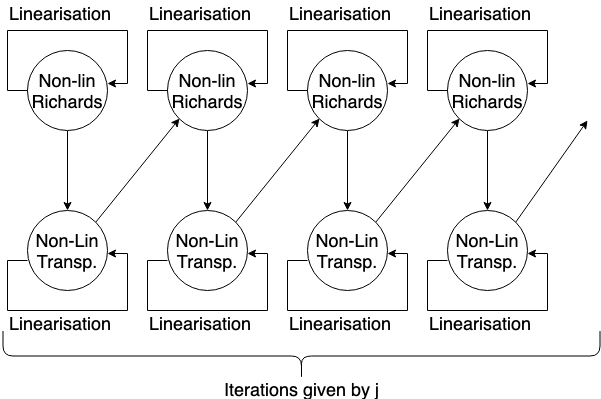}
\caption{The non-linear splitting approach}
\label{scheme2}
\end{figure}

\begin{figure}[H]
\centering
\includegraphics[scale=.4]{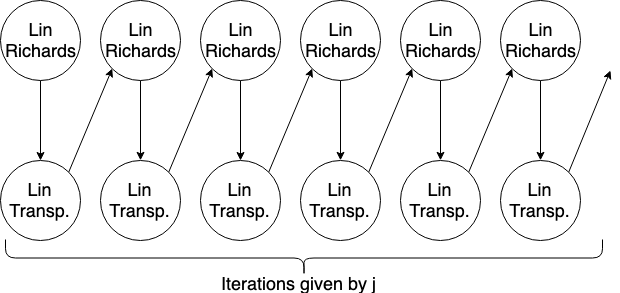}
\caption{The alternate splitting approach}
\label{scheme3}
\end{figure}

%We present here the following linearization schemes: Newton method, modified Picard and L-scheme.A new iteration, given by $j \in \mathbb{N}$, is introduced to study the non-linearity of $\theta$ and $K$. We denote with $\Psi^{n,j}$ the known solution at the time step $n$ and iteration $j$, while the unknown pressure will be $\Psi^{n,j+1}$. The iterations start with the solution computed at the previous time step, i.e. $\Psi^{n,0} = \Psi^{n-1}$. The loop will end whenever  $\norm{\Psi^{n,j+1}-\Psi^{n,j}} \leq \epsilon$, with $\epsilon$ a known constant indicating the required accuracy of the method and $||\cdot ||$ the Euclidean norm. Analogous requirements are made for the concentration $c$.

%\subsubsection*{Monolithic schemes}

\subsubsection{The monolithic Newton method (Mon-Newton)}
\paragraph{}
We recall that the Newton scheme is quadratically, but only locally convergent. The monolithic Newton method applied to \eqref{ric2}-\eqref{tra2} gives 
\par
\textbf{Problem PMon-Newton$_{n, j+1}$}: 
Let $\Psi^{n-1}_h, c^{n-1},\Psi^{n,j}_h, c^{n,j}_h\in V_h$ be given, 
\\find $\Psi^{n,j+1}_h,c^{n,j+1}_h \in V_h$ such that for all $v_h, w_h \in V_h$ one has
\begin{equation}\label{Newtonrichards}
\begin{split}
<\theta(\Psi^{n,j}_h, c^{n,j}_h) - \theta(\Psi^{n-1}_h, c^{n-1}_h), v_h > 
 + <\frac{\partial \theta}{\partial \Psi}(\Psi^{n,j}_h, c^{n,j}_h)
(\Psi^{n,j+1}_h-\Psi^{n,j}_h),v_h>
\\
 +\tau <K(\theta(\Psi^{n,j}_h, c^{n,j}_h)) (\nabla(\Psi^{n,j+1}_h)
+ \ez), \nabla v_h > \\
+\tau <\frac{\partial K}{\partial \Psi}(\theta(\Psi^{n,j}_h, c^{n,j}_h)) (\nabla(\Psi^{n,j+1}_h)+\ez)
 (\Psi^{n,j+1}_h-\Psi^{n,j}_h), \nabla v_h >
 =\tau <H,v_h>
\end{split}
\end{equation}
and 
\begin{equation}\label{Newtontransp}
\begin{split}
&< \theta(\Psi^{n,j}_h, c^{n,j}_h)c^{n,j+1}_h - \theta(\Psi^{n-1}_h, c^{n-1}_h)c^{n-1}_h, w_h > \\
&+<\frac{\partial \theta}{\partial c}(\Psi^{n,j}_h,c^{n,j}_h) (c^{n,j+1}_h- c^{n,j}_h), v_h> \\
&+ \tau < D\nabla c^{n,j+1}_h + \uwo c^{n,j+1}_h, \nabla w_h>   \ =\ \tau <H_c,w_h> .
\end{split}
\end{equation}

\subsubsection{The monolithic Picard method (Mon-Picard)}
\paragraph{}
The modified Picard method was initially proposed by Celia \cite{celia} for the Richards equation. It is similar to the Newton method in dealing with the nonlinearity in the saturation, but not in the permeability. Being a modification of the Newton method, modified Picard method is only linearly convergent \cite{pop2}. The monolithic Picard method applied to \eqref{ric2}-\eqref{tra2} becomes 
\par
\textbf{Problem PMon-Picard$_{n, j+1}$}: 
Let $\Psi^{n-1}_h,c^{n-1}_h,\Psi^{n,j}_h,c^{n,j}_h\in V_h$ be given, 
\\find $\Psi^{n,j+1}_h,c^{n,j+1}_h\in V_h$ such that for all $v_h, w_h \in V_h$ one has
\begin{equation}\label{picardrichards}
\begin{split}
&<\theta(\Psi^{n,j}_h,c^{n,j}_h) - \theta(\Psi^{n-1}_h,c^{n-1}_h), v_h > \\
&+ <\frac{\partial \theta}{\partial \Psi}(\Psi^{n,j}_h,c^{n,j}_h)
(\Psi^{n,j+1}_h-\Psi^{n,j}_h),v_h>\\
+ \tau <K(\theta(& \Psi^{n,j}_h,c^{n,j}_h)) (\nabla(\Psi^{n,j+1}_h) + \ez),\nabla v_h > \ =\ \tau <H,v_h >
\end{split}
\end{equation}
and
\begin{equation}\label{picardtransp}
\begin{split}
&< \theta(\Psi^{n,j}_h,c^{n,j}_h)c^{n,j+1}_h - \theta(\Psi^{n-1}_h,c^{n-1}_h)c^{n-1}_h, w_h > \\
& +  < \frac{\partial \theta}{\partial c}(\Psi^{n,j}_h,c^{n,j}_h) (c^{n,j+1}_h-c^{n,j}_h), w_h> \\
+ \tau < D\nabla & c^{n,j+1}_h + \uwo c^{n,j+1}_h, \nabla w_h> \ =\ \tau <H_c, w_h>.
\end{split}
\end{equation}

The equations (\ref{picardrichards}) and (\ref{picardtransp}) have been expressed as functions of only the unknown pressure $\Psi_h^{n,j+1}$ and concentration $c^{n,j+1}_h$, respectively. To achieve this, in the former we used only the derivative of $\theta$ with respect to $\Psi$ and only the derivative of $\theta$ with respect to $c$ in the latter.

Alternatively, both of the partial derivatives can be involved in the formulation,  
\begin{equation}
\begin{split}
\theta(\Psi^{n,j+1}_h, c^{n,j+1}_h) \rightarrow \theta(\Psi^{n,j}_h,c^{n,j}_h) &+ \Big(\frac{\partial \theta}{\partial \Psi}\Big)(\Psi^{n,j}_h, c^{n,j}_h)(\Psi^{n,j+1}_h - \Psi^{n,j}_h) \\
&+ \Big(\frac{\partial \theta}{\partial c}\Big)(\Psi^{n,j}_h, c^{n,j}_h)(c^{n,j+1}_h- c^{n,j}_h).
\end{split}
\end{equation}

\subsubsection{The monolithic L-scheme (Mon-LS)}
\paragraph{}
The monolithic L-scheme for solving \eqref{ric2}--\eqref{tra2} becomes 
\par
\textbf{Problem PMon-LS$_{n, j+1}$}: 
Let $\Psi^{n-1}_h,\Psi^{n,j}_h,c^{n-1}_h,c^{n,j}_h\in V_h$ be given and 
\\with $L_1, L_2>0$ large enough (as specified below), find $\Psi^{n,j+1}_h,c^{n,j+1}_h\in V_h$ s.t. for all $v_h, w_h \in V_h$ 
\begin{equation}\label{Lrichards}
\begin{split}
&<\theta(\Psi^{n,j}_h,c^{n,j}_h) - \theta(\Psi^{n-1}_h,c^{n-1}_h), v_h > + L_1<\Psi^{n,j+1}_h-\Psi^{n,j}_h,v_h> \\
&\tau <K(\theta(\Psi^{n,j}_h,c^{n,j}_h)) (\nabla(\Psi^{n,j+1}_h)+\ez),\nabla v_h > = \tau <H,v_h >, \qquad
\end{split}
\end{equation}
\begin{equation}\label{Ltransp}
\begin{split}
< \theta(\Psi^{n,j}_h,c^{n,j}_h)c^{n,j+1}_h - & \theta(\Psi^{n-1}_h,c^{n-1}_h)c^{n-1}_h, w_h > \\
+  L_2<c^{n,j+1}_h-c^{n,j}_h, & w_h> +\tau< D\nabla c^{n,j+1}_h 
+ \uwo c^{n,j+1}_h, \nabla w_h> \\
& =\tau<H_c,w_h> .
\end{split}
\end{equation}
$L_1$ and $L_2$ are free to be chosen parameters but should be large enough to ensure the convergence of the scheme, see Sec. \ref{sec:convergence}. In practice, the values of $L_1, L_2$ are connected to $ \displaystyle\max_\Psi \norm{\frac{\partial \theta}{\partial \Psi}} $, $\displaystyle\max_c \norm{\frac{\partial \theta}{\partial c}}$.

%As for the modified Picard method, an alternative scheme can be formulated by using
%\begin{equation}
%\theta(\Psi^{n,j+1}_h,c^{n,j+1}_h)\rightarrow \theta(\Psi^{n,j}_h,c^{n,j}_h) + L_1(\Psi^{n,j+1}_h-\Psi^{n,j}_h) + L_2(c^{n,j+1}_h-c^{n,j}_h).
%\end{equation}

The L-scheme does not involve the computations of derivatives, and the linear systems to be solved within each iteration are better conditioned compared to the ones given by Newton or Picard method (see \cite{List2016}). Moreover, this scheme is (linearly) convergent for any initial guess for the iteration.

%\subsection*{\large Iterative splitting schemes}

\subsubsection{The non-linear splitting approach (NonLinS)}
\paragraph{}
The non-linear splitting approach for solving \eqref{ric2}--\eqref{tra2} becomes 
\par
\textbf{Problem PNonLinS$_{n, j+1}$}: 
Let $\Psi^{n-1}_h, c^{n-1},\Psi^{n,j}_h, c^{n,j}_h\in V_h$ be given, find $\Psi^{n,j+1}_h \in V_h$ s.t. 
\begin{equation}\label{NonLinrichards}
\begin{split}
<\theta(\Psi^{n,j+1}_h, &c^{n,j}_h) - \theta(\Psi^{n-1}_h, c^{n-1}_h), v_h >  \\
+\tau <K(\theta(\Psi^{n,j}_h, c^{n,j}_h)&) (\nabla(\Psi^{n,j+1}_h)+ \ez), \nabla v_h >  
= \tau <H,v_h >
\end{split}
\end{equation}
holds true for all $v_h \in V_h$. Then, with $\Psi^{n,j+1}_h$ obtained, find $c^{n,j+1}_h \in V_h$ such that for all $w_h \in V_h$ it holds 
\begin{equation}\label{NonLintransp}
\begin{split}
< \theta(\Psi^{n,j+1}_h, &c^{n,j+1}_h)c^{n,j+1}_h - \theta(\Psi^{n-1}_h, c^{n-1}_h)c^{n-1}_h, w_h > 
+ \tau < D\nabla c^{n,j+1}_h \\
 &+\uwo c^{n,j+1}_h, \nabla w_h>   \ =\ \tau <H_c,w_h> .
\end{split}
\end{equation}
As for the monolithic schemes, one can apply the different linear iterative schemes to obtain fully linear versions of the splitting approach. This is done first to solve \eqref{NonLinrichards} and, once a solution to \eqref{NonLinrichards} is available, this is employed in the linearization of \eqref{NonLintransp}. 

\subsubsection{The alternate Newton method (AltS-Newton)}
\paragraph{}
In the alternate Newton method applied to \eqref{ric2}-\eqref{tra2} one solves 
\par
\textbf{Problem PAltS-Newton$_{n, j+1}$}: 
Let $\Psi^{n-1}_h, c^{n-1},\Psi^{n,j}_h, c^{n,j}_h\in V_h$ be given, 
\\find $\Psi^{n,j+1}_h \in V_h$ s.t. 
\begin{equation}\label{AltNewtrichards}
\begin{split}
&<\theta(\Psi^{n,j}_h, c^{n,j}_h) - \theta(\Psi^{n-1}_h, c^{n-1}_h), v_h > \\
+< \theta'(\Psi^{n,j}_h, c^{n,j}_h) (& \Psi^{n,j+1}_h-\Psi^{n,j}_h), v_h> 
+\tau <K(\theta(\Psi^{n,j}_h, c^{n,j}_h)) (\nabla(\Psi^{n,j+1}_h)\\
& + \ez), \nabla v_h >  
+\tau <\frac{\partial K}{\partial \Psi}(\theta(\Psi^{n,j}_h, c^{n,j}_h)) (\nabla(\Psi^{n,j+1}_h)\\
& +\ez)(\Psi^{n,j+1}_h-\Psi^{n,j}_h), \nabla v_h >
 = \tau <H,v_h >
\end{split}
\end{equation}
holds true for all $v_h \in V_h$. Then,  with $\Psi^{n,j+1}_h$ obtained above, find $c^{n,j+1}_h \in V_h$ such that for all $w_h \in V_h$ one has 
\begin{equation}\label{AltNewttransp}
\begin{split}
< \theta(\Psi^{n,j+1}_h, &c^{n,j}_h)c^{n,j+1}_h - \theta(\Psi^{n-1}_h, c^{n-1}_h)c^{n-1}_h, w_h > \\
+< \frac{\partial \theta}{\partial c}(\Psi^{n,j+1}_h, &c^{n,j}_h)(c^{n,j+1}_h- c^{n,j}_h), v_h>  + \tau < D\nabla c^{n,j+1}_h \\
& +\uwo c^{n,j+1}_h, \nabla w_h>   \ =\ \tau <H_c,w_h> .
\end{split}
\end{equation}

\subsubsection{The alternate Picard method (AltS-Picard)}
\paragraph{}
The alternate Picard method applied to \eqref{ric2}-\eqref{tra2} becomes 
\par
\textbf{Problem PAltS-Picard$_{n, j+1}$}: 
Let $\Psi^{n-1}_h, c^{n-1},\Psi^{n,j}_h, c^{n,j}_h\in V_h$ be given, 
\\find $\Psi^{n,j+1}_h \in V_h$ s.t. 
\begin{equation}\label{Altpicardrichards}
\begin{split}
&<\theta(\Psi^{n,j}_h,c^{n,j}_h) - \theta(\Psi^{n-1}_h,c^{n-1}_h), v_h > \\
&+ <\frac{\partial \theta}{\partial \Psi}(\Psi^{n,j}_h,c^{n,j}_h)(\Psi^{n,j+1}_h-\Psi^{n,j}_h),v_h>\\
 + \tau <K(\theta(\Psi^{n,j}_h,& c^{n,j}_h)) (\nabla(\Psi^{n,j+1}_h) + \ez),\nabla v_h > \ =\ \tau <H,v_h >
\end{split}
\end{equation}
hold true for all $v_h \in V_h$. Then,  with $\Psi^{n,j+1}_h$ obtained above, find $c^{n,j+1}_h \in V_h$ such that for all $w_h \in V_h$ one has 
\begin{equation}\label{Altpicardtransp}
\begin{split}
< \theta(\Psi^{n,j+1}_h,&c^{n,j}_h)c^{n,j+1}_h - \theta(\Psi^{n-1}_h,c^{n-1}_h)c^{n-1}_h, w_h > \\
&+  <\frac{\partial \theta}{\partial c}(\Psi^{n,j+1}_h,c^{n,j}_h)(c^{n,j+1}_h-c^{n,j}_h),w_h> \\
+ \tau < D\nabla &c^{n,j+1}_h + \uwo c^{n,j+1}_h, \nabla w_h> \ =\ \tau <H_c, w_h>.
\end{split}
\end{equation}

\subsubsection{The alternate L-scheme (AltS-LS)}
\paragraph{}
The alternate L-scheme for solving (\ref{ric1}--\ref{tra1}) becomes 
\par
\textbf{Problem PAltS-LS$_{n, j+1}$}: 
Let $\Psi^{n-1}_h, c^{n-1},\Psi^{n,j}_h, c^{n,j}_h\in V_h$ be given, find $\Psi^{n,j+1}_h \in V_h$ s.t. 
\begin{equation}\label{Lrichards}
\begin{split}
&<\theta(\Psi^{n,j}_h,c^{n,j}_h) - \theta(\Psi^{n-1}_h,c^{n-1}_h), v_h > + L_1<\Psi^{n,j+1}_h-\Psi^{n,j}_h,v_h> \\
&\tau <K(\theta(\Psi^{n,j}_h,c^{n,j}_h)) (\nabla(\Psi^{n,j+1}_h)+\ez),\nabla v_h >\ = \tau <H,v_h >
\end{split}
\end{equation}
hold true for all $v_h \in V_h$. Then,  with $\Psi^{n,j+1}_h$ obtained above, find $c^{n,j+1}_h \in V_h$ such that for all $w_h \in V_h$ one has 
\begin{equation}\label{Ltransp}
\begin{split}
< \theta(\Psi^{n,j+1}_h, c^{n,j}_h)& c^{n,j+1}_h - \theta(\Psi^{n-1}_h,c^{n-1}_h)c^{n-1}_h, w_h >\\
 +  L_2<c^{n,j+1}_h-&c^{n,j}_h,w_h> +\tau< D\nabla c 
+ \uwo c^{n,j+1}_h, \nabla w_h> \\
& =\ \tau <H_c,w_h> . 
\end{split}
\end{equation}

\begin{remark} (Stopping criterion) For both monolithic and splitting schemes, one stops the iteration process whenever 
$$ 
\norm{\Psi^{n,j+1}_h-\Psi^{n,j}_h} \leq  \epsilon_1,  \text{ and }
\norm{c^{n,j+1}_h-c^{n,j}_h} \leq \epsilon_2,
$$
where $\epsilon_1, \epsilon_2$ are small numbers. Here we took $10^{-07}$ or $10^{-08}$.
\end{remark}

\section{Convergence analysis}\label{sec:convergence}
\paragraph{}
In this section we analyse the convergence of the monolithic L-scheme introduced through Problem PMon-LS$_{n, j+1}$. We restrict the analysis to this iteration, but mention that the convergence analysis for the other (monolithic or splitting) schemes introduced above, can be done in a similar fashion. We start by defining the errors 
\begin{equation}
e_{\Psi}^{j+1}\ :=\Psi^{n,j+1}_h - \Psi^{n,j}_h \ \text{and} \  e_{c}^{j+1}\ :=c^{n,j+1}_h - c^{n,j}_h ,
\end{equation}
obtained at iteration $j+1$. The scheme is convergent if both errors vanish when $j\rightarrow\infty$.

The convergence is obtained under the following assumptions: 
\begin{itemize}
%\item[(A1)]{There exist $\alpha_\Psi > 0$ and $\alpha_c \ge 0$ such that there holds for any $\Psi_1, \Psi_2,  \in \mathbb{R}$  and $c_1, c_2 \in \mathbb{R}_+$
%\begin{equation} \label{assumption_theta}
%\begin{split}
%  <\theta(\Psi_1, c_1)  -  & \theta(\Psi_2, c_2), \Psi_1 - \Psi_2> + <c_1 \theta(\Psi_1, c_1)  \\
%  - c_2 \theta(\Psi_2, c_2), & c_1 - c_2> \, 
%  \ge \alpha_\Psi \norm{\theta(\Psi_1, c_1)  -  \theta(\Psi_2, c_2)}^2 \\
% & + \alpha_c \norm{\Psi_1 - \Psi_2}^2.
%\end{split}
%\end{equation}
%Furthermore, there exist two constants $\theta_m \ge 0$ and $\theta_M< \infty$ such that $\theta_m \leq \theta \leq\theta_M$.}
\item[(A1)]{There exist $\alpha_\Psi > 0$ and $\alpha_c \ge 0$ such that 
for any $\Psi_1, \Psi_2 \in \mathbb{R}$  and $c_1, c_2 \in \mathbb{R}_+$ 
\begin{equation} \label{assumption_theta}
\begin{split}
	<\theta(\Psi_1, c_1)  -  & \theta(\Psi_2, c_2), \Psi_1 - \Psi_2> + 
	<c_1 \theta(\Psi_1, c_1)  - c_2 \theta(\Psi_2, c_2),  c_1 - c_2> \, \\
	& \ge \alpha_\Psi \norm{\theta(\Psi_1, c_1)  -  \theta(\Psi_2, c_2)}^2
	+ \alpha_c \norm{\Psi_1 - \Psi_2}^2.
\end{split}
\end{equation}
Furthermore, there exist two constants $\theta_m \ge 0$ and $\theta_M < \infty$ such that $\theta_m \leq \theta \leq\theta_M$.}

\item[(A2)]{The function $K(\theta(\cdot , \cdot ))$ is Lipschitz continuous, with respect to both variables, and there exist two constants $K_m$ and $K_M$ such that $0\leq K_m\leq K\leq K_M<\infty$.}
\item[(A3)]{There exist $M_u, M_\Psi, M_c \ge 0$  such that \\$\norm{\uwn}_{L^\infty} \le M_u$, $\norm{\nabla \Psi^{n}}_{L^\infty}\leq M_\Psi $ and $\norm{c^{n}}_{L^\infty} \le M_c$  for all $n \in \mathbb{N}$.}
\end{itemize}

\begin{remark} (A2) is satisfied in most realistic situations. (A3) is a pure technical one, being satisfied when data is sufficiently regular, which is assumed to be the case for the present analysis. The inequality \eqref{assumption_theta} in (A1) is a coercivity assumption. It is in particular satisfied if $\theta$ only depends on $\Psi$, and for common relationships $\theta -- \Psi$ encountered in the engineering literature.
\end{remark}

\begin{theorem}
Let $n\ \in \ \{1,2, \dots N\}$ be given and assume (A1)-(A3) be satisfied. If the time step is small enough (see \eqref{conditions_tau} below), the monolithic L-scheme  in (\ref{Lrichards}) and (\ref{Ltransp}) is linearly convergent for any $L_1$ and $L_2$ satisfying \eqref{Lcondtions}. 
\end{theorem}

\begin{proof}
We follow the ideas in \cite{Pop2004,List2016} and start by subtracting (\ref{ric2}) from (\ref{Lrichards}) to obtain the error equation
\begin{align}
\begin{split}
<\theta^{n,j}_h-& \theta^{n}_h,v_h> + L_1<\Psi^{n,j+1}_h-\Psi^{n,j}_h,v_h> \\
+ \tau <K^{n,j}_h\nabla e_{\Psi}^{n,j+1},\nabla v_h>  + & \tau <(K^{n,j}
-K^{n})\nabla \Psi^{n,j+1}_h,\nabla v_h> \\
+ \tau & <(K^{n,j}-K^{n})\ez, \nabla v_h> = 0.
\end{split}
\end{align}
Testing now the above equation with $v_h = e_{\Psi}^{j+1}$, one obtains
\begin{align}
\begin{split}
& <\theta^{n,j}_h-\theta^{n}_h,e_{\Psi}^{j+1}> + L_1<e_{\Psi}^{j+1}-e_{\Psi}^{j},e_{\Psi}^{j+1}> \\
+ \tau <K^{n,j}\nabla e_{\Psi}^{n,j+1},& \nabla e_{\Psi}^{j+1}>  + \tau <(K^{n,j}_h
-K^{n}_h)\nabla \Psi^{n,j+1}_h,\nabla e_{\Psi}^{j+1}> \\
& + \tau <(K^{n,j}_h-K^{n}_h)\ez,\nabla e_{\Psi}^{j+1}> = 0.
\end{split}
\end{align}
By (A2) and after some algebraic manipulations we further get
\begin{align}\label{p1}
\begin{split}
& <\theta^{n,j}_h-\theta^{n}_h,e_{\Psi}^{j}  > +  \frac{L_1}{2}\norm{e_{\Psi}^{j+1}}^2 + \frac{L_1}{2}\norm{e_{\Psi}^{j+1} - e_{\Psi}^{j}}^2 \\
& + \tau K_m \norm{\nabla e_{\Psi}^{j+1}}^2 
  \leq \frac{L_1}{2}\norm{e_{\Psi}^{j}}^2 - <\theta^{n,j}_h-\theta^{n}_h,e_{\Psi}^{j+1} - e_{\Psi}^{j}> \\
  - \tau <(& K^{n,j}_h- K^{n}_h )\nabla \Psi^{n,j+1}_h,\nabla e_{\Psi}^{j+1}> 
 - \tau <(K^{n,j}_h -K^{n}_h)\ez ,\nabla e_{\Psi}^{j+1}>.
\end{split}
\end{align}
Using now (A1), (A3), the Lipschitz continuity of $K$ and twice the Young and Cauchy-Schwarz inequalities, for any $\delta_0 > 0$ and $\delta_1 > 0$, from \eqref{p1} one obtains 
\begin{align}\label{richardsfinal}
\begin{split}
& <\theta^{n,j}_h-\theta^{n}_h,e_{\Psi}^{j}> + \frac{L_1}{2}\norm{e_{\Psi}^{j+1}}^2 + \frac{L_1}{2} \norm{e_{\Psi}^{j+1} - e_{\Psi}^{j}}^2 \\
+  \tau K_m \norm{\nabla e_{\Psi}^{j+1}}^2 & \leq \frac{L_1}{2}\norm{e_{\Psi}^{j}}^2 
+ \frac{\delta_0}{2} \norm{ \theta^{n,j}_h - \theta^{n}_h }^2 + \frac{1}{2 \delta_0}\norm{e_{\Psi}^{j+1} - e_{\Psi}^{j}}^2 \\
 & + \frac{\tau (M^2_\Psi+1)L_k^2}{2\delta_1}  \norm{ \theta^{n,j}_h - \theta^{n}_h }^2 +  \tau \delta_1\norm{\nabla e_{\Psi}^{j+1}}^2. 
\end{split}
\end{align}
Similarly, subtracting (\ref{tra2}) from (\ref{Ltransp}) and choosing $w_h = e_{c}^{j+1}$ in the resulting one gets
\begin{equation}
\begin{split}
<c^{n,j+1}_h\theta^{n,j} _h - c^{n}_h\theta^{n}_h, e_{c}^{j+1}>  + L_2<e_{c}^{j+1}-e_{c}^{j},e_{c}^{j+1}> \\
+ \tau <  D\nabla e_{c}^{j+1}+ \uwo e_{c}^{j+1}, \nabla e_{c}^{j+1} > =\ 0.
\end{split}
\end{equation}
This can be rewritten as
\begin{align}\label{eq_proof_1}
\begin{split}
<c^{n,j}_h\theta^{n,j}_h  - c^{n}\theta^{n}_h, e_{c}^{j}>  + <\theta^{n,j}_h e_{c}^{j+1}, e_{c}^{j+1}>  + \frac{L_2}{2}\norm{e_{c}^{j+1}}^2 \\
 + \frac{L_2}{2}\norm{e_{c}^{j+1} - e_{c}^{j}}^2 + \tau D <\nabla e_{c}^{j+1} , \nabla e_{c}^{j+1}>  =\ \frac{L_2}{2}\norm{e_{c}^{j}}^2 \\
 + <\theta^{n}_h c^{n}_h - \theta^{n,j}_h c^{n, j}_h, e_{c}^{j+1} - e_{c}^{j}> - \tau  < \uwo e_{c}^{j+1}, \nabla e_{c}^{j+1} >.
 \end{split}
\end{align}
Using again (A1), (A3)  and the Cauchy-Schwarz and Young inequalities, from \eqref{eq_proof_1} it follows that for any $\delta_2,\delta_3,  \delta_4 > 0$ one has
\begin{align}\label{transpfinal}
\begin{split}
<c^{n,j}_h\theta^{n,j}_h  - c^{n}_h\theta^{n}_h, e_{c}^{j}>  + \theta_m \norm{e_{c}^{j+1}}^2
 + \frac{L_2}{2}\norm{e_{c}^{j+1}}^2 + \frac{L_2}{2}\norm{e_{c}^{j+1} - e_{c}^{j}}^2 \\
 + \tau D \norm{\nabla e_{c}^{j+1}}^2 
   \le \ \dfrac{L_2}{2}\norm{e_{c}^{j}}^2 + 
 \dfrac{\delta_2}{2} \norm{\theta^{n}_h - \theta^{n,j}_h}^2 + \dfrac{\delta_3}{2} \norm{e_{c}^{j}}^2 \\ 
 +(\dfrac{M_c^2}{ 2 \delta_2}  + \dfrac{\theta_M^2}{ 2 \delta_3}) \norm{e_{c}^{j+1} - e_{c}^{j}}^2  + \tau  \dfrac{M_u^2}{ 2 \delta_4} \norm{e_{c}^{j+1}}^2 +  \tau \dfrac{\delta_4}{2} \norm{\nabla e_{c}^{j+1}}^2.
  \end{split}
  \end{align}
Summing adding (\ref{richardsfinal}) to (\ref{transpfinal}) and using (A1) one gets
\begin{align}\label{eq_proof_2}
\begin{split}
 \alpha_\Psi   \norm{\theta^{n}_h - \theta^{n,j}_h}^2 + \frac{L_1}{2} \norm{e_{\Psi}^{j+1}}^2 + \frac{ L_1}{2} \norm{e_{\Psi}^{j+1} - e_{\Psi}^{j}}^2  + \tau K_m \norm{\nabla e_{\Psi}^{j+1}}^2 \\
 + \alpha_c \norm{e_{c}^{j}}^2 + \theta_m \norm{ e_{c}^{j+1}}^2
 + \frac{L_2}{2}\norm{e_{c}^{j+1}}^2 + \frac{L_2}{2}\norm{e_{c}^{j+1} - e_{c}^{j}}^2 \\
 + \tau D \norm{\nabla e_{c}^{j+1}}^2   \le \  \frac{L_1}{2}\norm{e_{\Psi}^{j}}^2 + ( \frac{\delta_0}{2} + \frac{\tau (M^2_\Psi+1)L_k^2}{2\delta_1} \\
 + \dfrac{\delta_2}{2})\norm{ \theta^{n,j}_h - \theta^{n}_h  }^2 + \frac{1}{2 \delta_0}\norm{e_{\Psi}^{j+1} - e_{\Psi}^{j}}^2  +  \tau \delta_1\norm{\nabla e_{\Psi}^{j+1}}^2  + \dfrac{L_2}{2}\norm{e_{c}^{j}}^2  \\
 + \dfrac{\delta_3}{2} \norm{e_{c}^{j}}^2  +(\dfrac{M_c^2 }{ 2 \delta_2}  + \dfrac{\theta_M^2}{ 2 \delta_3}) \norm{e_{c}^{j+1} - e_{c}^{j}}^2  + \tau  \dfrac{M_u^2}{ 2 \delta_4} \norm{e_{c}^{j+1}}^2 \\
 +  \tau \dfrac{\delta_4}{2} \norm{\nabla e_{c}^{j+1}}^2.
\end{split}
\end{align}
Choosing $\delta_0 = \delta_2 = \dfrac{\alpha_\Psi}{2}$, $\delta_1 = \dfrac{K_m}{2}$, $\delta_3 = \theta_m$ and $\delta_4 = \dfrac{D}{2}$ in \eqref{eq_proof_2}, and assuming that 
\begin{equation} \label{Lcondtions}
L_1 \ge  \dfrac{2}{\alpha_\Psi} \; \text{ and } \; L_2 \ge \dfrac{2 M_c^2}{ \alpha_\Psi } +  \dfrac{\theta^2_M}{ \theta_m}, 
\end{equation}
and the time step $\tau$ satisfies the mild conditions
\begin{equation}\label{conditions_tau}
\alpha_\Psi - 2 \tau \frac{\tau (M^2_\Psi+1)L_k^2}{K_m} \ge 0 \quad \rm{ and } \quad \theta_m + 2 \alpha_c + \dfrac{\tau D}{C_\Omega} -  \dfrac{ 2 \tau M^2_u}{D} \ge 0,
\end{equation}
where  $C_\Omega$ denotes the Poincare constant, then we obtain
\begin{align}\label{eq_proof_3}
\begin{split}
  \frac{L_1}{2}  \norm{e_{\Psi}^{j+1}}^2 +  \tau \dfrac{K_m}{2} \norm{\nabla e_{\Psi}^{j+1}}^2  
 + (\frac{L_2}{2} + \theta_m - \tau  \dfrac{M_u^2}{ D} ) \norm{e_{c}^{j+1}}^2 \\
 + \tau \dfrac{D }{2} \norm{\nabla e_{c}^{j+1}}^2  
  \le  \frac{L_1}{2}\norm{e_{\Psi}^{j}}^2  + (\dfrac{L_2}{2} + \dfrac{\theta_m}{2}- \alpha_c )\norm{e_{c}^{j}}^2 . 
 \end{split}
\end{align}
Finally, by using the Poincare inequality two times we get from \eqref{eq_proof_3}
\begin{align}\label{eq_proof_4}
\begin{split}
 (L_1 +  \tau \dfrac{K_m}{C_\Omega}) & \norm{e_{\Psi}^{j+1}}^2  + (L_2 +  2 \theta_m + \tau \dfrac{D }{C_\Omega}- 2 \tau  \dfrac{ M_u^2}{ D} ) \norm{e_{c}^{j+1}}^2  \\
 & \le  L_1\norm{e_{\Psi}^{j}}^2  + (L_2 + \theta_m - 2 \alpha_c )\norm{e_{c}^{j}}^2 . 
 \end{split}
\end{align}
From \eqref{conditions_tau}, \eqref{eq_proof_4} implies that the errors are contracting and therefore the monolithic L-scheme (\ref{Lrichards}) - (\ref{Ltransp}) is convergent. 
\end{proof}

\begin{remark}
The convergence rate resulting from (\ref{eq_proof_4}) does not depend on the spatial mesh size. Also, observe that this convergence is obtained for any initial guess. Based on this, the method is globallly convergent, which is in contrast to the Newton or (modified) Picard schemes, converging only locally. It can be observed that, the larger the time step, the smaller the constants $L_1$ and $L_2$ are, resulting in a faster convergence. For small steps instead the convergence rate can approach 1. On the other hand, if the time step is small enough, one may reach the regime where the Newton scheme becomes convergent (see \cite{pop2}). Alternatively, one may first perform a number of L-scheme iterations, and use the resulting as an initial guess for the Newton scheme (see \cite{List2016}), or consider the modified L-scheme in \cite{Mitra}. In either situations, the convergence behaviour was much improved. 
\end{remark}

\begin{remark}
The convergence of the modified Picard and Newton method applied to the Richards equation has been already proved in \cite{pop2}. Such results can be extended to the coupled problems considered here.
\end{remark}

\section{Numerical examples}\label{numericalexample}
\paragraph{}
In this section we consider four test cases for the proposed linearization schemes, inspired by the literature \cite{List2016,surf}. The schemes have been implemented in the open source software package MRST \cite{mrst}, an open source toolbox based on Matlab, in which multiple solvers and models regarding flows in porous media are incorporated.
\begin{figure}[h!]
\begin{center}
\begin{subfigure}{.45\textwidth}  
  \includegraphics[width=1\linewidth]{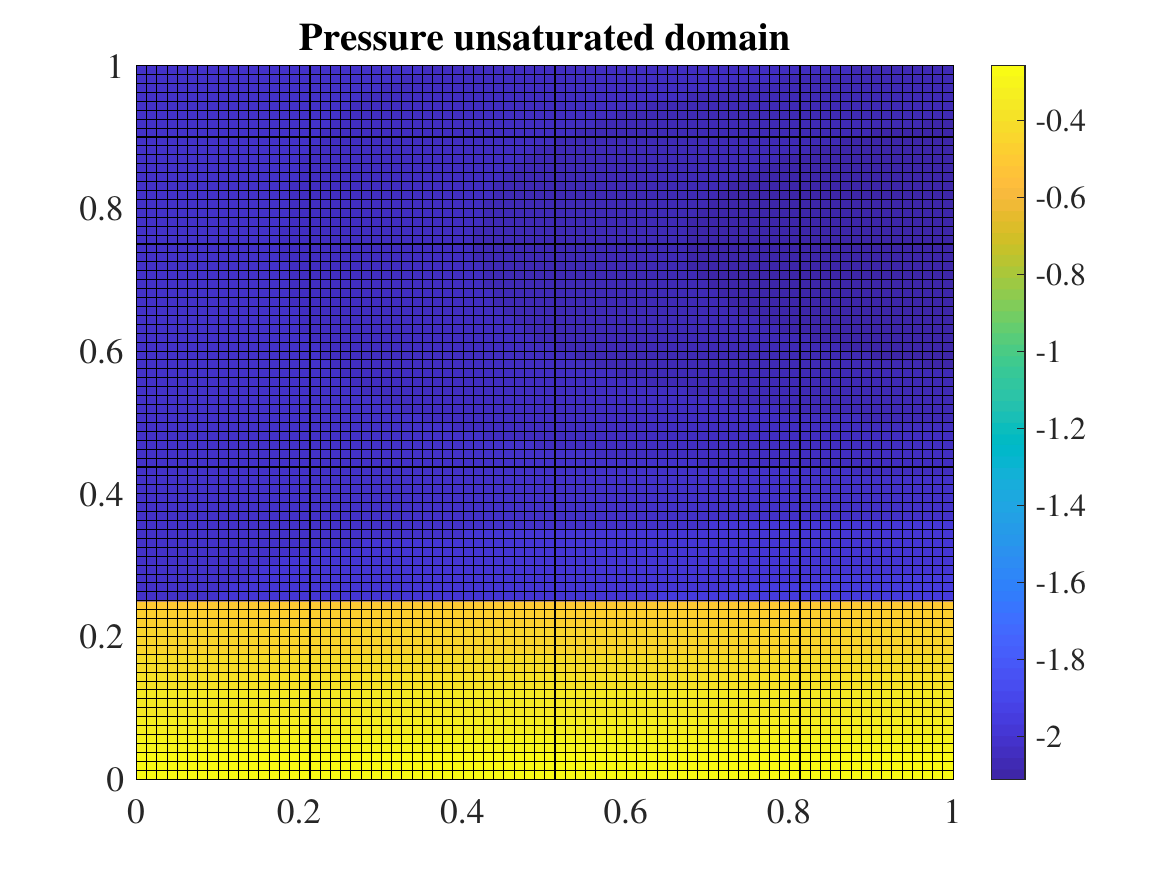}
  \caption{Pressure profile at T = 1}
  \label{fig:PressureVadose}
\end{subfigure}
  \begin{subfigure}{.45\textwidth}  \includegraphics[width=1\linewidth]{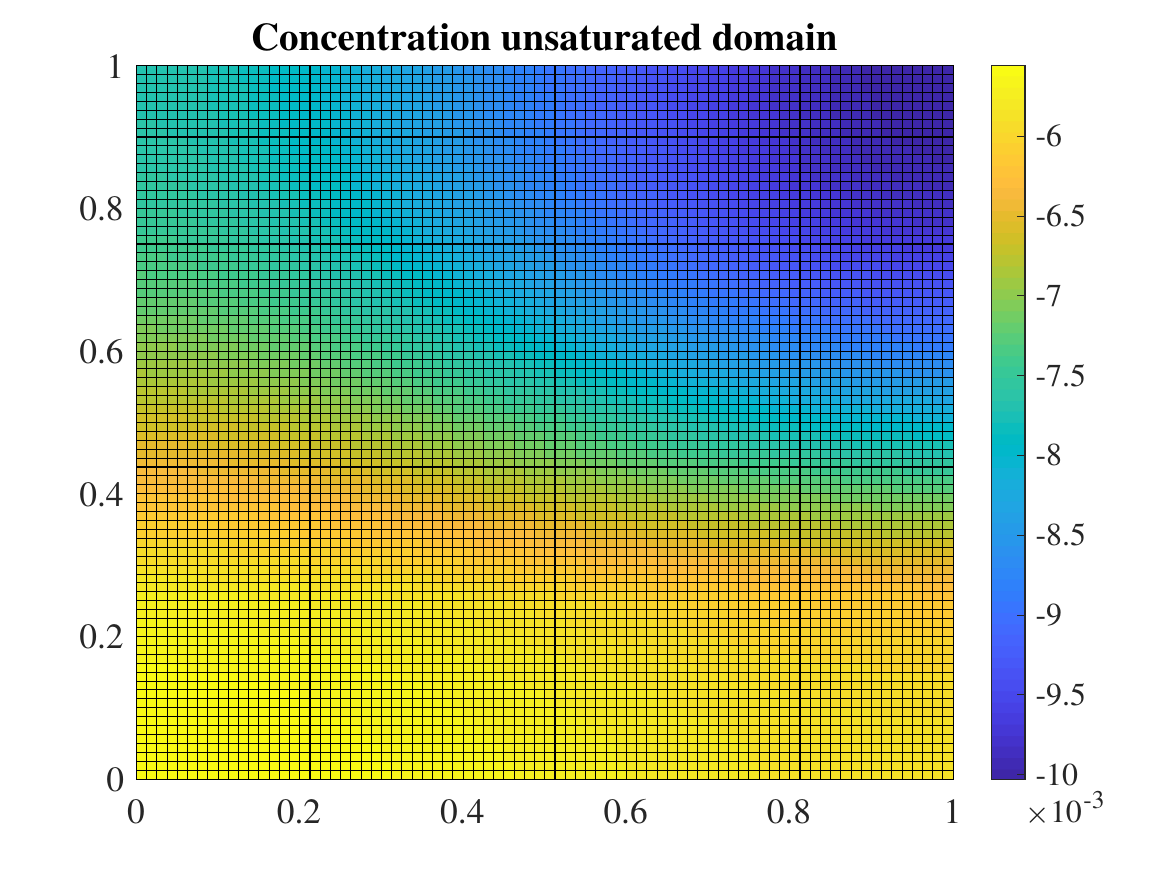}
  \caption{Concentration profile at T = 1}
  \label{fig:ConcentrationVadose}
\end{subfigure}
\caption{Example 1A: pressure and concentration profiles at the final time $T=1$. The simulations were performed with $dx = 1/80$ and $\tau =1/10$}
\end{center}
\end{figure}

\begin{figure}[h!]
\begin{center}
  \includegraphics[scale =.35]{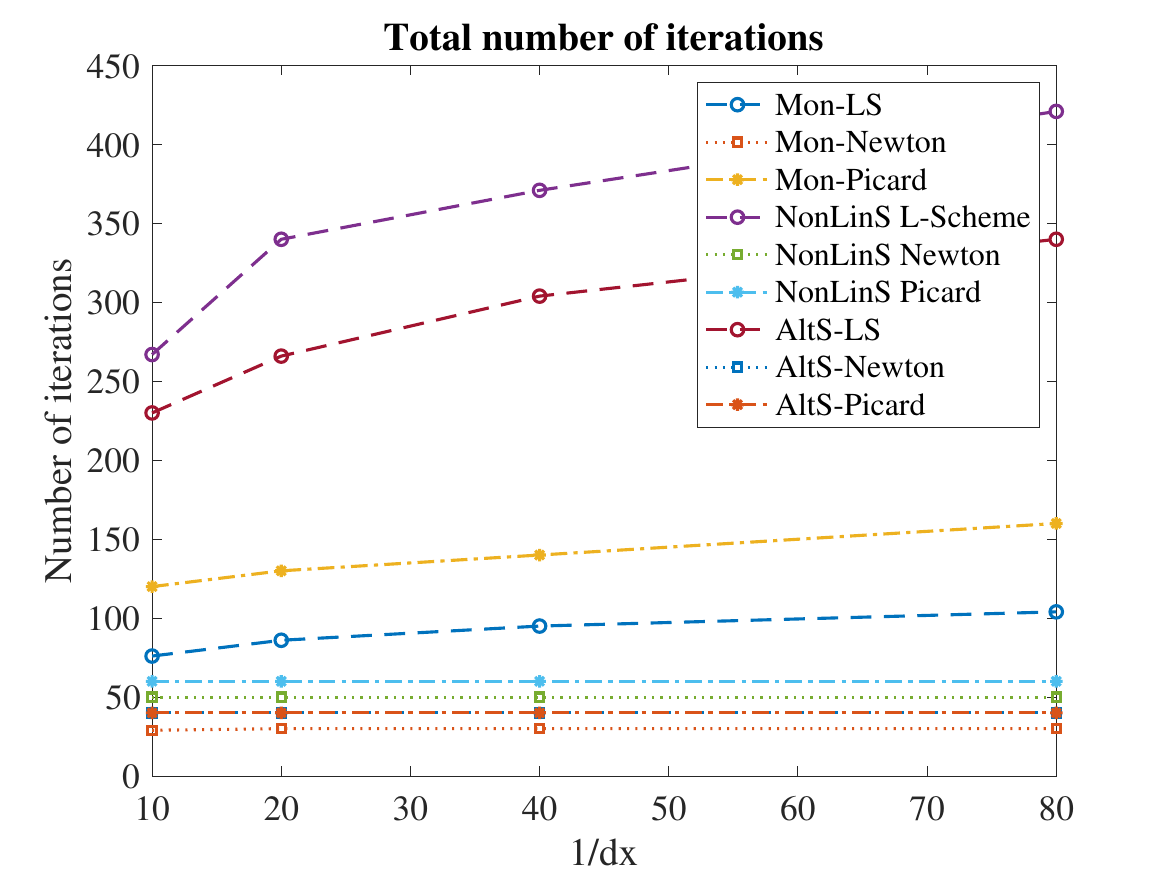}
  \caption{Total numbers of iterations}
  \label{fig:NumberIterationsDifferentSchemes}
\end{center}
\end{figure}

\subsection*{Example 1A: flow and transport in strictly unsaturated media}
\paragraph{}
We start our numerical studies with a manufactured problem admitting an analytical solution \cite{List2016}. The unit square $\Omega$ is divided into two sub-domains: $\Omega_{up}$ and $\Omega_{down}$. The two regions are defined as: $\Omega_{up} = [0,1] \times [1/4,1]$ and $\Omega_{down}=[0,1] \times [0, 1/4]$.
Dirichlet boundary conditions, $\Psi = -3$, and no-flow Neumann boundary conditions are imposed on $\Gamma_{D} = [0,1] \times {1}$ and $\Gamma_{N} = \partial \Omega / \Gamma_D $, respectively. A constant initial pressure $p^0_{up} = -2$, and a non-constant $p^0_{down} = - y - 1/4$ are defined in the upper and in the lower part of the domain, $\Omega_{up}$ and $\Omega_{down}$. The van Genuchten parameters are presented in Table \ref{tab:parameterstab}.

Further, for both Richards and transport equations, we have a source term, $f(x,y) =.006\cos(4/3\pi y)\sin(x)$, on $\Omega_{up}$. No external forces or sources, are defined in the lower region, i.e. $f = 0$ on $\Omega_{down}$. Finally, the initial condition for the concentration is given by  $c^0 = 1$ and the boundary conditions by $c_{| {\Gamma_D}} = 4$. 
 
 \begin{table}
\centering
\begin{tabular}{ c c }
\hline \hline
&  \\ [0.5ex]
 $T_{max}$ & 10 h \\
 $\Omega$  & $[0  ,1 ] \times [0 , 1]$  \\
 $L_\Psi$ & .1 \\
 $L_c$ & .005\\
 Van Genuchten parameters &  \\ 
 $\theta_s$ & .42 \\
 $\theta_r$ & .0026 \\
 $n$ & 2.9 \\
 $\alpha$ & .95 \\ 
 $a$ & .044 \\
 $b$ & .4745 \\ 
 Surface tension parameters & \\
 \textbf{$\zeta$} & 2.4901 \\
 \textbf{$\sigma0$} & 73 mN/m \\
 $K_s$ & .12 cm/min  \\
 $D_0$ & 6.0e-04 \\
 Accuracy requirement & \\
 $\epsilon$ & e-07 \\
\hline
\end{tabular}
\caption{Parameters involved in all the examples}
\label{tab:parameterstab}
\end{table}

We performed simulations using regular meshes, consisting of squares, whose sides were of length $dx = [1/10,\ 1/20,\ 1/40,\ 1/80 ]$.  We consider also varying time steps  of sizes $\tau = [1/10,\ 1/20,\ 1/40,\ 1/80 ]$. In  Fig. \ref{fig:PressureVadose} we are plotting the pressure and concentration profiles at the final time $T = 1$. We point out that in this first example we are always in the strictly unsaturated regime, implying that Richards equation is a regular. All the proposed iterative schemes were converging for this example. In Fig. \ref{fig:NumberIterationsDifferentSchemes} is given the total number of iterations for the different schemes.

%For the $L$-scheme we tested multiple values of $L_{\Psi}$ and $L_c$. The best results, in terms of numbers of iterations, are achieved for $L_{\Psi}\approx \frac{1}{2} \displaystyle\sup_\Psi \mid \partial \theta / \partial \Psi\mid$, analogous for $L_c$. Such results had already been observed into \cite{List2016}.

%The previous Figure (\ref{fig:NumberIterationsDifferentSchemes}) regards the different numbers of iterations required by each linearization scheme and coupling algorithm. It is possible to observe that, the $L$-scheme presents the highest number of iterations.

%Equally accurate results are obtained, with fewer iterations, thanks to the Newton and modified Picard methods. Figure (\ref{fig:NumberIterationsDifferentSchemesZoomes}) gives us a closer comparison. For this particular problem, the Monolithic Newton scheme (MON-Newton) is the least requiring in terms of iterations. Although, in case of larger domains and more refined meshes, the matrices obtained with the monolithic approach can result so large to give computational problems, such as: high condition numbers and large computational times.\par 

More details regarding the total number  of iterations and the condition number of the linear systems are presented in Tables \ref{tab:1.0dx}, \ref{tab:1.0dt}. The condition number is computed at the first iteration of each algorithm and with respect to the Euclidean norm. In Table \ref{tab:1.0dx}, we fixed a time step $\tau=1/10$ and we investigated different mesh sizes, precisely $dx = [1/10, 1/20, 1/40, 1/80]$. In Table \ref{tab:1.0dt} we use a constant $dx = 1/40$ and varying the time step sizes $ \tau = [1/10, 1/20, 1/40, 1/80]$. We point out that the alternate schemes are converging much faster than the classical ones. We also remark the high differences in the condition numbers, the L-scheme based algorithms being much better conditioned.

%to have more detailed information, regarding the numbers of iteration and the condition numbers of the systems associated to each solving algorithm. In the former, Table (\ref{tab:1.0dx}), we fixed a time step $\tau=1/10$ and we investigated different mesh sizes, precisely $dx = [1/10, 1/20, 1/40, 1/80]$. In the latter we use a constant $dx = 1/40$ and varying time steps $ \tau = [1/10, 1/20, 1/40, 1/80]$. \parWe can observe once more the difference in the numbers of iterations, already presented into Figure (\ref{fig:NumberIterationsDifferentSchemes}), having here more precise details. It is also interesting the comparison between the condition numbers of the systems associated to each algorithm. In case of the splitting approaches, solving the two equation separately, we have two systems and two condition numbers. As stated previously one of the main advantages of implementing a modified Picard or an $L$-scheme, instead of the classical Newton method, is the better conditioning of the linearized system. Such improvement  is evident for each solving algorithm. \par The results presented into the second table (\ref{tab:1.0dt}) are similar to the ones in the former. We can observe as smaller time steps give us better conditioned linearized systems. Once more it is evident a reduction in the number of iterations between the classical (NonLinS) and the alternate (AltLinS) splitting approaches.
%Saturated porous media fixed dt
\begin{table}[]
\centering
  \scalebox{.6}{
  \begin{tabular}{ c | c c | c c  c c  | c c c c}
\hline \hline
& Monolithic & &  & NonLinS &  & & & AltLinS &&\\
\hline
& Newton & & & Newton &  & & &Newton  & &\\
\hline
&  & & &  &  cond. \# & & &  & cond. \# &\\
\hline
dx &\# iterations & condition \#  & \# iterations & Richards && Transport  & \# iterations & Richards && Transport\\ [0.5ex]
 1/10 & 20 & 511.0045         & 40 & 333.4035  && 5.9916 & 20 & 333.4019 && 5.9916  \\ 
 1/20 & 20 & 2.2933e+03   & 40 & 1.5040e+03 && 6.2079 & 20 & 1.5040e+03& & 6.2079 \\
 1/40 & 20 & 9.4458e+03  & 40  & 6.1312e+03 && 6.3234 & 20 & 6.1312e+03 && 6.3234\\
 1/80 & 20 & 3.8371e+04   & 40 &  2.4774e+04 & & 6.3816 & 20 & 2.4774e+04 &&  6.3817\\
 \hline \hline
 & L Scheme & & & L Scheme &  & & & L Scheme & & \\
 \hline
 &  & & &  &  cond. \# & & &  & cond. \# &\\
\hline
dx &\# iterations & condition \#  & \# iterations & Richards && Transport  & \# iterations & Richards && Transport\\ [0.5ex]
1/10 & 277 & 183.4223  & 540 & 177.4742   &&  2.1356  &  264  & 177.4725   &&  2.1356   \\ 
 1/20 & 300 & 812.5650 & 650  &   796.5765  && 2.1839  &  316  &   796.5755   &&  2.1839  \\
 1/40 & 363 & 3.3450e+03 & 750  &  3.2584e+03 & & 2.2092   & 368   & 3.2584e+03 &  & 2.2092 \\
 1/80 & 510 & 1.3585e+04  & 850  & 1.3191e+04  && 2.2220   &  421  &  1.3191e+04   && 2.2220 \\
 \hline\hline
 & Picard & & & Picard   & & & & Picard & &\\
 \hline
 &  & & &  &  cond. \# & & &  & cond. \# &\\
\hline
dx & \# iterations & condition \#  & \# iterations & Richards && Transport  & \# iterations & Richards && Transport\\ [0.5ex]
 1/10 & 100 & 326.8280  &  40  &  177.4610    &&  5.9916   & 20  & 177.4601    && 5.9916  \\
 1/20 & 110 & 1.4667e+03  & 40  &  796.5170 &&  6.2079 &  20 &  796.5129 && 6.2079 \\
 1/40 & 120 & 6.0380e+03  &  40 &  3.2581e+03 &&  6.3234  &  20 & 3.2581e+03 & & 6.3234 \\
 1/80 & 130 &  2.4522e+04   &  40 & 1.3190e+04  &&  6.3816 &  20 &    1.3190e+04 &&   6.3817\\
 \hline
\end{tabular}
}
\caption{Example 1A: unsaturated medium, fixed $\tau=1/10$}
\label{tab:1.0dx}
\end{table}
%Saturated porous media fixed dx=1/40
\begin{table}[]
\centering
    \scalebox{.6}{
  \begin{tabular}{ c | c c | c c  c c  | c c c c}
\hline \hline
& Monolithic & &  & NonLinS &  & & & AltLinS &&\\
\hline
& Newton & & & Newton &  & & &Newton  & &\\
\hline
&  & & &  &  cond. \# & & &  & cond. \# &\\
\hline
$\tau$ &\# iterations & condition \#  & \# iterations & Richards && Transport  & \# iterations & Richards && Transport\\ [0.5ex]
 1/10 & 20 & 9.4458e+03  & 40 & 6.1312e+03    && 6.3234 & 20 & 6.1312e+03  && 6.3234  \\ 
 1/20 & 40 & 4.7275e+03   & 80 & 3.2581e+03 && 6.3234 & 40 &  3.2580e+03 &&  6.3234 \\
 1/40 & 80 &  2.3677e+03  & 160  & 1.7024e+03 && 6.3234 & 80 & 1.7024e+03 &&  6.3234\\
 1/80 & 160 &  1.1876e+03  & 320 &   870.8016 &&   6.3234& 160 &  870.8010 &&  6.3234 \\
  \hline \hline
 & L Scheme & & & L Scheme &  & & & L Scheme & & \\
 \hline
&  & & &  &  cond. \# & & &  & cond. \# &\\
\hline
$\tau$ &\# iterations & condition \#  & \# iterations & Richards && Transport  & \# iterations & Richards && Transport\\ [0.5ex]
1/10 & 363 & 3.3450e+03 & 750    & 3.2584e+03  &&  2.2092   &  368  & 3.2584e+03  &&  2.2092   \\ 
 1/20 & 570 & 1.7540e+03 & 1300  &  1.7026e+03  && 2.2092    &  633  &  1.7026e+03  &&  2.2092  \\
 1/40 & 1048 & 898.9759  & 2160  & 870.2808       && 2.2092    & 1050  &  870.8979    && 2.2092 \\
 1/80 & 1914 &  455.3332  & 3520 &  440.6573      && 2.2092    &  1700 &  440.8161    && 2.2092 \\
 \hline\hline
 & Picard & & & Picard   & & & & Picard & &\\
 \hline
&  & & &  &  cond. \# & & &  & cond. \# &\\
\hline
$\tau$ &\# iterations & condition \#  & \# iterations & Richards && Transport  & \# iterations & Richards && Transport\\ [0.5ex]
 1/10 & 120 & 6.0380e+03  &  40  &  3.2581e+03  &&  6.3234 & 20  & 3.2581e+03  &&6.3234 \\
 1/20 & 220 & 3.0216e+03  & 80  &  1.7025e+03  &&  6.3234  &  40 &  1.7025e+03 && 6.3234 \\
 1/40 & 400 & 1.5132e+03 &  160 &  870.8263  &&  6.3234     &  80  & 870.8251  && 6.3234 \\
 1/80 & 640 &  758.9936   &  320 &  440.8018   && 6.3234     &  160 &  440.8015 &&  6.3234\\
 \hline
\end{tabular}
}
\caption{Example 1A: unsaturated medium, fixed dx=1/40}
\label{tab:1.0dt}
\end{table}

\subsection*{Example 1B: flow and transport in variably saturated porous media}\label{example2}
\paragraph{}
For the second example we use the same domain, mesh sizes, boundary conditions and parameters, but we allow a saturated/unsaturated regime by changing the initial condition for the pressure. We consider a subdivision of $p^0$ between upper and lower regions, precisely: $p^0_{up} = -2$ and $p^0_{down} = -y + 1/4$. This new expression for $p^0_{down}$ gives a positive pressure in the lower part of the domain (saturated region).  For this example the Richards equation is now degenerate parabolic, therefore more challenging for the numerical schemes. Furthermore, we introduce this time also a  reaction term $R(c)$ in the transport equation, given by $ R(c)\ :=\ c/(1+c) $.

At the iteration $j+1$, the term $R(c)$ is linearized in the following way:
\begin{equation*}
R(c^{n+1,j+1}) \rightarrow \frac{1+c^{n+1,j+1}}{c^{n+1,j}}
\end{equation*}
\begin{figure}[]
\begin{center}
\begin{subfigure}{.45\textwidth}  
  \includegraphics[width=1\linewidth]{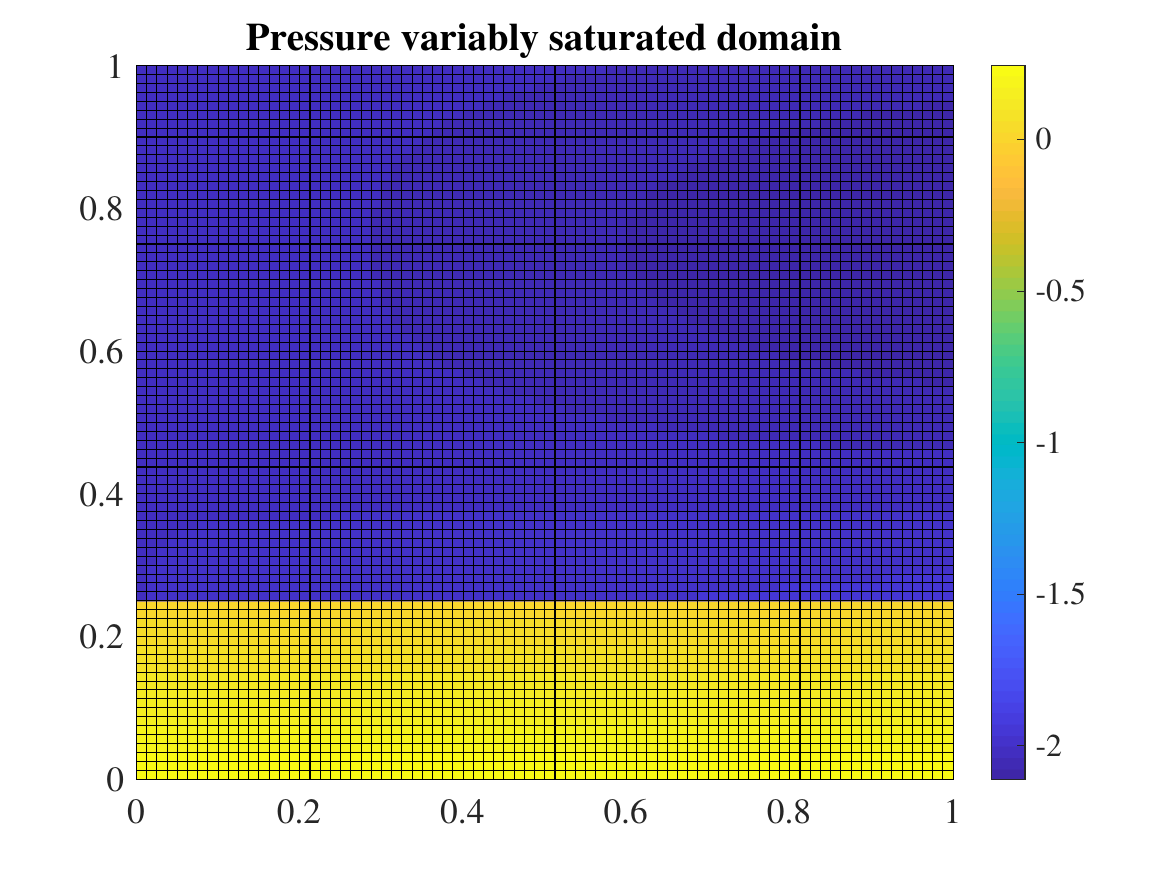}
  \caption{Pressure profile}
  \label{fig:variablypressure}
\end{subfigure}
  \begin{subfigure}{.45\textwidth}
  \includegraphics[width=1\linewidth]{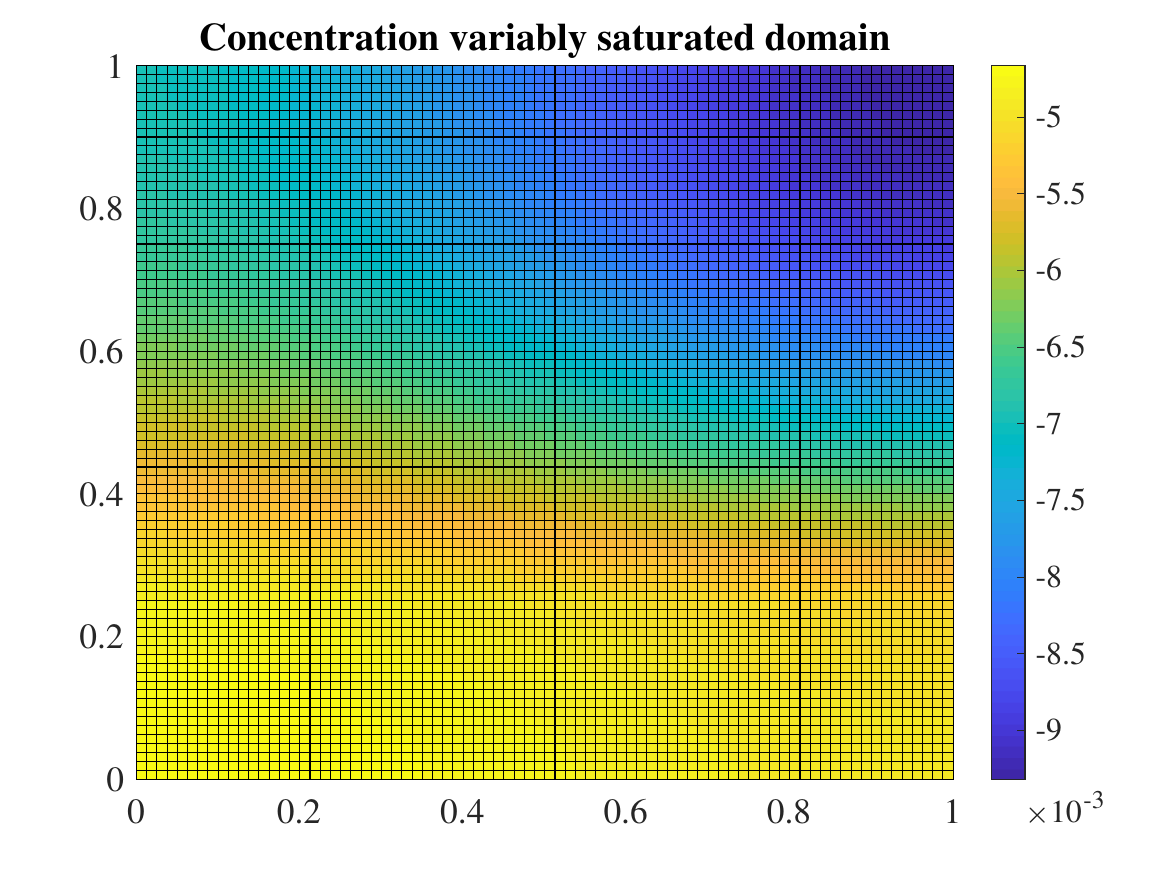}
  \caption{Concentration profile}
  \label{fig:variablyconcentration}
\end{subfigure}
\caption{Example 1B: plots of pressure and concentration in the variably saturated medium, the simulations were done with $dx = 1/80$ and $\tau =1/10$}
\end{center}
\end{figure}
\begin{figure}[h!]
\centering
  \includegraphics[scale =.35]{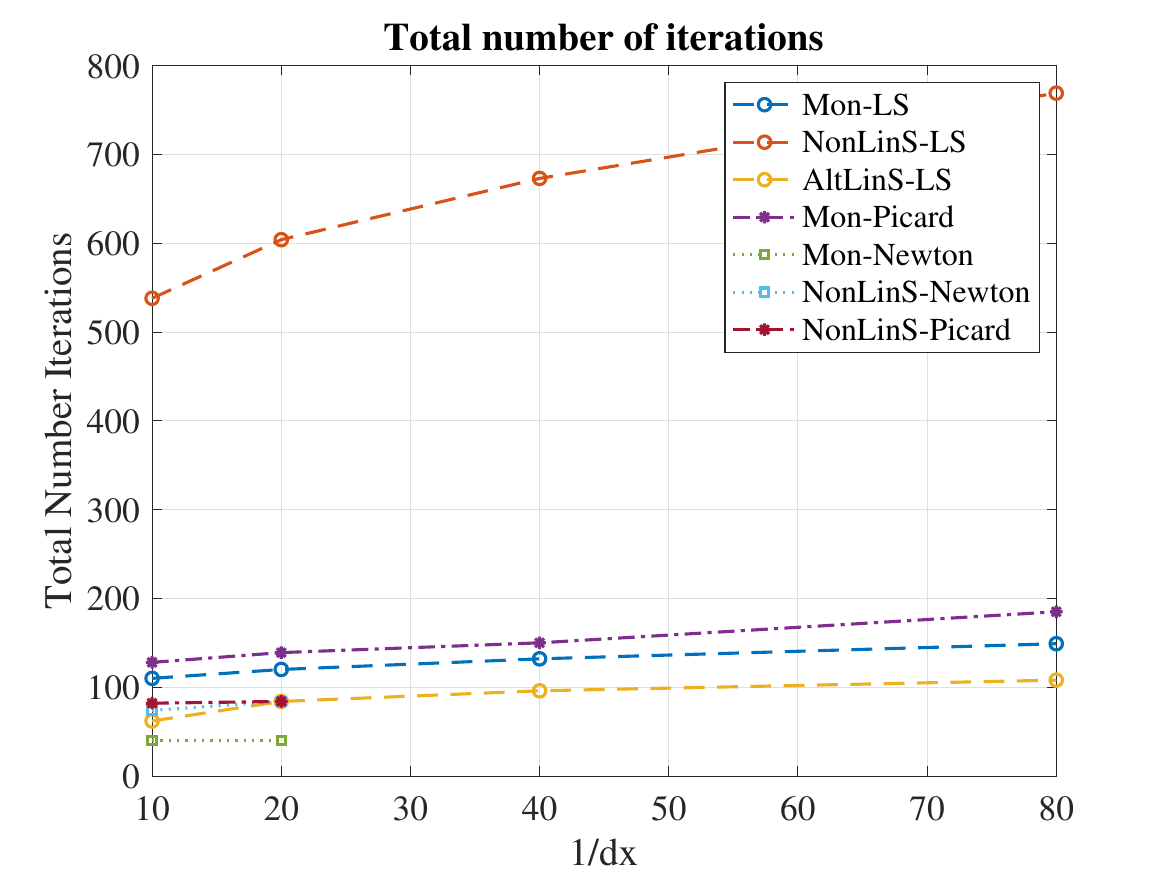}
  \caption{Example 1B: numbers of iterations in the variably saturated porous medium}
  \label{fig:iterationvariably}
\end{figure}

In Fig. \ref{fig:variablypressure} we show again the pressure and concentration profiles at the final step $T=1$.    The main differences to the previous example, i.e. Figure (\ref{fig:PressureVadose}) are in the values of the pressure. We can observe again a discontinuity in the pressure profile but, more importantly, it is evident a jump from negative to positive values. Such results were expected considering the initial pressure imposed on the domain.\par

In Fig. \ref{fig:iterationvariably} are presented the total number of iterations. We remark that in this case only the L-scheme based algorithms are converging. It is also interesting to observe that the difference in the number of iterations between the more commonly used non-linear splitting approach (NonLinS) and alternate splitting (AltLinS) approach. The alternate method appears to be a valid alternative to the common formulation. It produces equally accurate results, requiring fewer iterations.

%We proceed again in a comparison of the different numbers of iterations obtained for each linearization scheme and algorithm implemented. Observing Figure (\ref{fig:iterationvariably}) we notice complete different results compared to the one presented into Figure (\ref{fig:NumberIterationsDifferentSchemes}). For this particular case, due to the complexity of the problem, only the $L$-scheme converges for every algorithm implemented. Such scheme is known for its robustness and for the better conditioning of the linearized systems associated with it. The modified Picard scheme failed to converge in the splitting approaches, the derivatives of $\theta$ where producing badly conditioned system. The Newton method, presented in the previous example as the best linearization schemes, in terms of number of iterations, does not converge in case of too refined meshes. Reducing the time mesh size has improved only partially the scheme.
%\par 
%It is also interesting to observe, from Figure (\ref{fig:iterationvariably}), the difference in the number of iterations between the more commonly used non-linear splitting approach (NonLinS) and alternate splitting (AltLinS) approach. Once more, the algorithm presented in the paper, appears to be a valid alternative to the common formulation. It produce equally accurate results, requiring fewer iterations.
\par 
As for the previous example we present in the Tables \ref{tab:1.1dx}, \ref{tab:1.1dt} the precise numbers of iterations and condition numbers for each algorithm implemented for different mesh diameters and time steps. Each segment ($-$) in the tables below, implies that the method failed to converge for such particular combination of time step and space mesh. As already observed in Fig. \ref{fig:iterationvariably} the L-scheme based solvers are the only ones converging in all cases. Moreover, the linear systems associated with the L-scheme are better conditioned than the ones for Picard or Newton methods. \textcolor{blue}{We finally remark that, as expected, for smaller time steps  the Newton and Picard schemes converge, see Table \ref{tab:1.1dt}. Anyhow, the condition numbers of the systems associated to the linearized problems, remain ill-conditioned. This can cause further numerical problems.}
% Variably saturated media, dt = 1/10
\begin{table}[]
\centering
    \scalebox{.6}{
  \begin{tabular}{ c | c c | c c  c c  | c c c c}
\hline \hline
& Monolithic & &  & NonLinS &  & & & AltLinS &&\\
\hline
& Newton & & & Newton &  & & &Newton  & &\\
\hline
&  & & &  &  cond. \# & & &  & cond. \# &\\
\hline
dx &\# iterations & condition \#  & \# iterations & Richards && Transport  & \# iterations & Richards && Transport\\ [0.5ex]
 1/10 & 28 & 2.2753e+11             & 50 & 5.7734e+09            && 1.1251                & - & 5.5783e+09 && 1.0117 \\ 
 1/20 & - &  1.2345e+12 & - & 4.6521e+09 &&  1.0126 & - & .6521e+09 &&  1.0126  \\
 1/40 & - & 4.5159e+12 & -  &5.2321e+09 &&  1.0124 & - & 5.2321e+09 && 1.0124 \\
 1/80 & - & .7232e+13  & - & 5.5219e+09  & & 1.0123 & - & 5.5219e+09  && 1.0123\\
 \hline \hline
 & L Scheme & & & L Scheme &  & & & L Scheme & & \\
 \hline
 &  & & &  &  cond. \# & & &  & cond. \# &\\
\hline
dx &\# iterations & condition \#  & \# iterations & Richards && Transport  & \# iterations & Richards && Transport\\ [0.5ex]
 1/10 & 175 & 247.8672 & 440        & 239.2940 && 1.3314  &  264  & 239.8408   &&  1.3314   \\ 
 1/20 & 314 & 1.0576e+03  & 650  &  1.0432e+03  &&  1.3338&  316  &  1.0432e+03  && 1.3338  \\
 1/40 & 352 &  4.2256e+03  &  750 &  4.1291e+03  && 1.3328  & 368   & 4.1291e+03 && 1.3328 \\
 1/80 & 408& 1.6902e+04   &  910  & 1.6437e+04  &&  1.3323  &  421  &  1.6437e+04   &&  1.3323 \\
 \hline\hline
 & Picard & & & Picard   & & & & Picard & &\\
 \hline
 &  & & &  &  cond. \# & & &  & cond. \# &\\
\hline
dx &\# iterations & condition \#  & \# iterations & Richards && Transport  & \# iterations & Richards && Transport\\ [0.5ex]
 1/10 & - &4.5478e+11  & 50 & 5.7735e+09  && 1.1251                           & - & 5.5783e+09 &&.0117 \\ 
 1/20 & - & 2.4690e+12 & - & 4.6521e+09&&  1.0126 & - & 4.6521e+09&&1.0126  \\
 1/40 & - &  9.0318e+12  & -  &5.2321e+09&&   1.0124 & - &  5.2321e+09 && 1.0124 \\
 1/80 & - &  3.4465e+13   & - &  5.5219e+09 &&   1.0123& - &  5.5219e+09 && 1.0123\\
 \hline
\end{tabular}
}
\caption{Example 1B: variably saturated medium, fixed $\tau=1/10$}
\label{tab:1.1dx}
\end{table}
%Variably saturated porous media fixed dx=1/40
\begin{table}[]
\centering
  \scalebox{.6}{
  \begin{tabular}{ c | c c | c c  c c  | c c c c}
\hline \hline
& Monolithic & &  & NonLinS &  & & & AltLinS &&\\
\hline
& Newton & & & Newton &  & & &Newton  & &\\
\hline
&  & & &  &  cond. \# & & &  & cond. \# &\\
\hline
$\tau$ &\# iterations & condition \#  & \# iterations & Richards && Transport  & \# iterations & Richards && Transport\\ [0.5ex]
1/10 & - & 4.5159e+12 & - & 5.2321e+09  && 1.0124          & - & 5.5783e+09 && 1.0117 \\ 
 1/20 & - & 4.5194e+12 & - & 2.1747e+10&& 1.0062           & - & 4.6521e+09  && 1.0126  \\
 1/40 & 80 &  4.5265e+12& 200  &1.0442 e+10  & & 1.1325& - & 5.2321e+09 & &1.0124 \\
 1/80 & 160 & 4.5406e+12& 400 & 4.3494e+10 && 1.1325 & - & 5.5219e+09 & &1.0123  \\
 \hline \hline
 & L Scheme & & & L Scheme &  & & & L Scheme & & \\
 \hline
 &  & & &  &  cond. \# & & &  & cond. \# &\\
\hline
$\tau$ &\# iterations & condition \#  & \# iterations & Richards && Transport  & \# iterations & Richards && Transport\\ [0.5ex]
 1/10 & 352  & 4.2256e+03 & 750    & 4.1291e+03 &&  1.3328  &  368  & 4.1291e+03  & &  1.3328   \\ 
 1/20 & 627 & 2.2518e+03 & 1300  &  2.1862e+03  && 1.3328  &  633  &  2.1890e+03   &&  1.3328  \\
 1/40 &1100 & 1.1624e+03 & 2160  & 1.1258e+03   && 1.3328  & 1050  &  1.1266e+03 & &1.3328 \\
 1/80 & 1900 & 589.7690 & 3520 &  570.6119      && 1.3328     &  1700 &  571.0523 && 1.3328 \\
 \hline\hline
 & Picard & & & Picard   & & & & Picard & &\\
 \hline
&  & & &  &  cond. \# & & &  & cond. \# &\\
\hline
$\tau$ &\# iterations & condition \#  & \# iterations & Richards && Transport  & \# iterations & Richards && Transport\\ [0.5ex]
 1/10 & -& 2.4690e+12 &  - &  5.2321e+09  &&  1.0124 & -  & 5.2321e+09 &&1.0124\\
 1/20 & - & 9.0388e+12  & -  & 1.0442e+10  &&  1.0062  &- & 1.0442e+10  && 1.0062  \\
 1/40 & - & 9.0529e+12  & 200 & 2.1748e+10  &&  1.1325 &-  & 2.0860e+10  &&1.0031  \\
 1/80 & - & 9.0811e+12  &  400 & 4.3494e+10  && 1.1325 &- & 4.1698e+10  && 1.0015 \\
 \hline
\end{tabular}
}
\caption{Example 1B: variably saturated medium, fixed dx=1/40}
\label{tab:1.1dt}
\end{table}
\subsection*{Example 2A: well in unsaturated porous media}
\paragraph{}
Our next example is inspired from \cite{surf}.  We consider same domain (e.g. the unit square), boundary and initial condition and parameters as in the first numerical example (1A). The medium is again strictly unsaturated. We include now, in the upper part of the domain, a well and inject water with a specific concentration of the external component. No analytical solution is available for this case. Due to the higher complexity of the problem we use more refined meshes, precisely $dx = [1/50, 1/100, 1/150, 1/200]$. The pressure at the well is set to $p_W = -10$ and the concentration of the surfactant to $c_W = 10$. 
\par 
In Fig. \ref{fig:wellunsat}, we present the different profiles of pressure and concentration at the initial time $t^0$ and at final time $T=1$ day. 
\begin{figure}[H]
\begin{center}
\begin{subfigure}{.45\textwidth}  
  \includegraphics[width=1\linewidth]{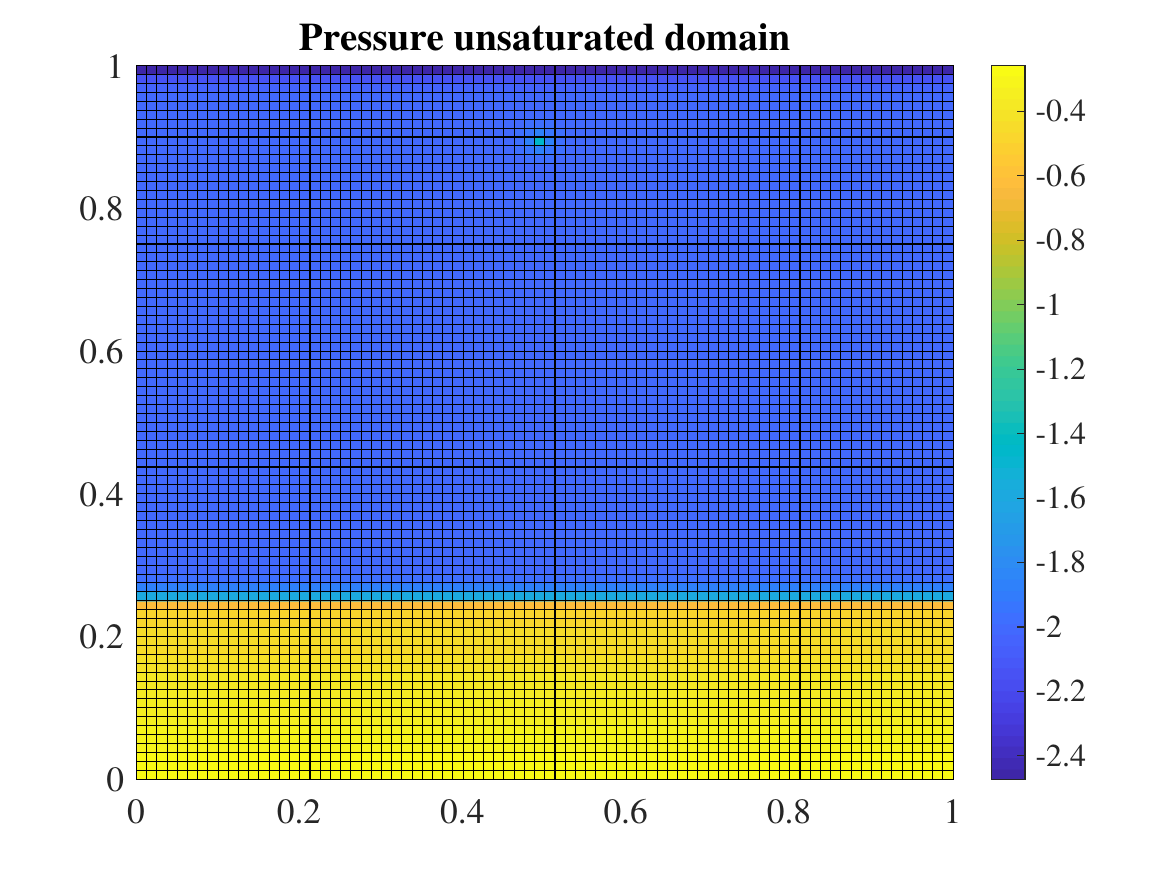}
  \caption{Pressure profile at first time step}
  \label{fig:variablypressure}
\end{subfigure}
  \begin{subfigure}{.45\textwidth}
  \includegraphics[width=1\linewidth]{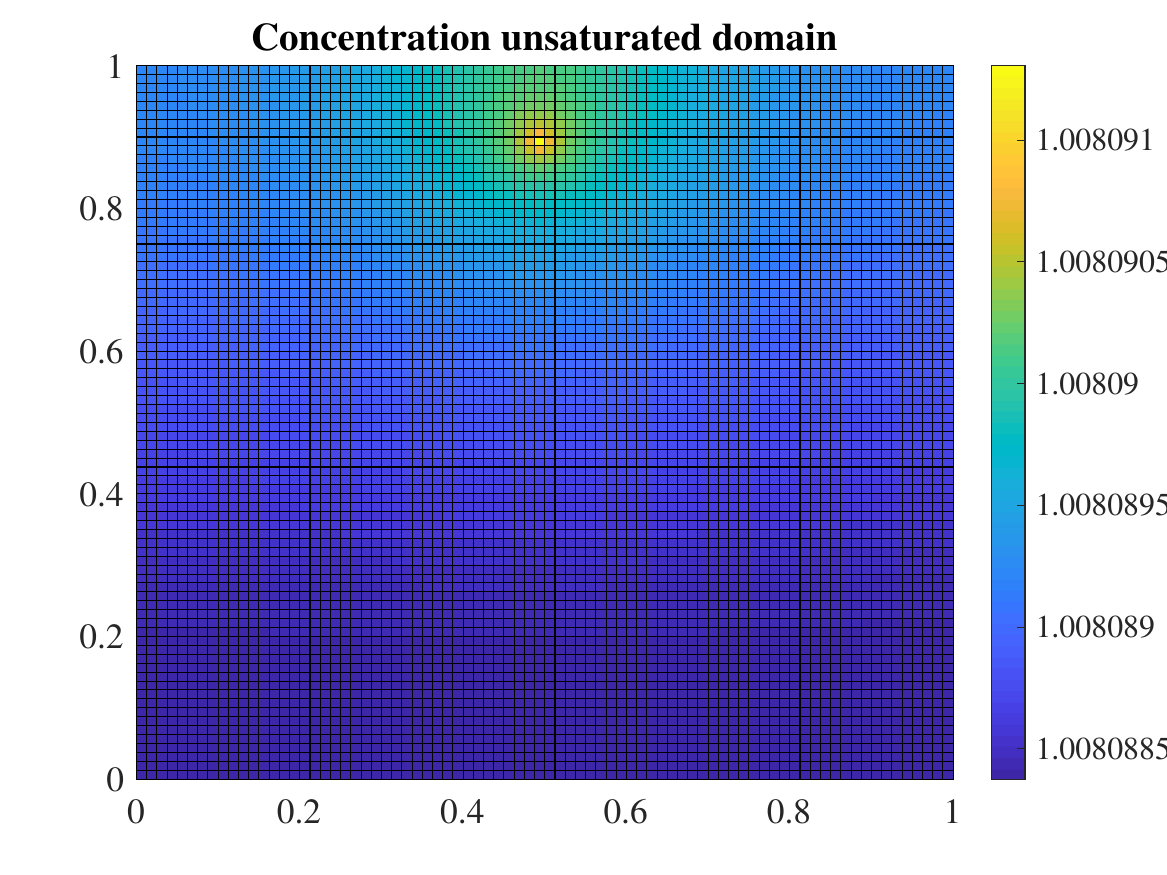}
  \caption{Concentration profile at first time step}
  \label{fig:variablyconcentration}
\end{subfigure}
  \begin{subfigure}{.45\textwidth}  
  \includegraphics[width=1\linewidth]{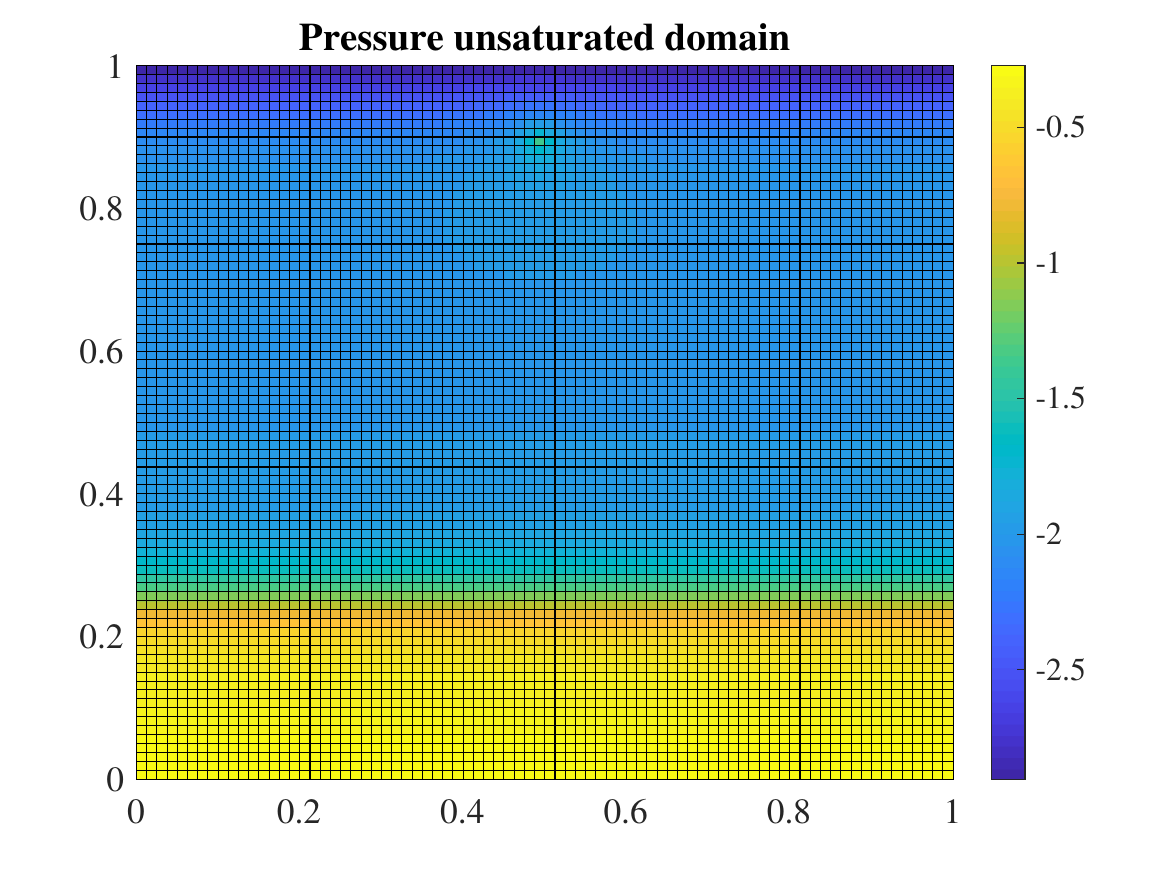}
  \caption{Pressure profile after one day}
  \label{fig:variablypressure}
\end{subfigure}
  \begin{subfigure}{.45\textwidth}
  \includegraphics[width=1\linewidth]{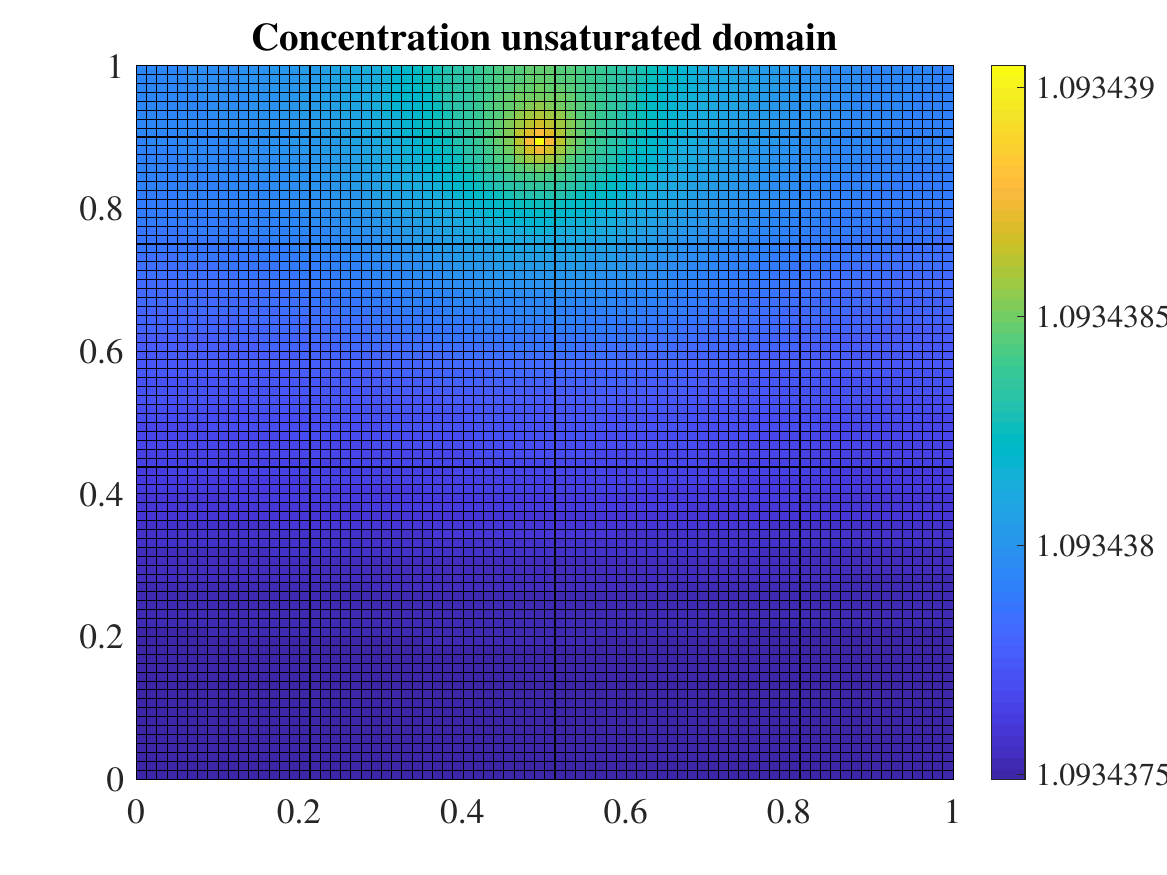}
  \caption{Concentration profile after one day}
  \label{fig:variablyconcentration}
\end{subfigure}
\caption{Example 2A: plots of pressure and concentration in unsaturated medium, the simulations were done with $dx = 1/100$ and $\tau =dx$}
\label{fig:wellunsat}
\end{center}
\end{figure}

Once more, in Fig. \ref{fig:iterationwell} we compare the different solving algorithms. We study the numbers of iterations and the conditions numbers of the linearized systems. As for the first example, the media being unsaturated, the Richards equation does not degenerate and all the schemes converge. 
We can observe, in the Tables \ref{tab:2.0dx}, \ref{tab:2.0dt}, that the monolithic Newton method is the fastest, in term of numbers of iterations. We remark that the alternate splitting approach (AltLinS), once more, requires fewer iterations than the classical splitting algorithm (NonLinS) for all the linearization schemes. The linear systems resulting by applying the L-scheme based solvers are better conditioned compared with the other solvers.
\begin{figure}[]
\centering
  \includegraphics[scale =.4]{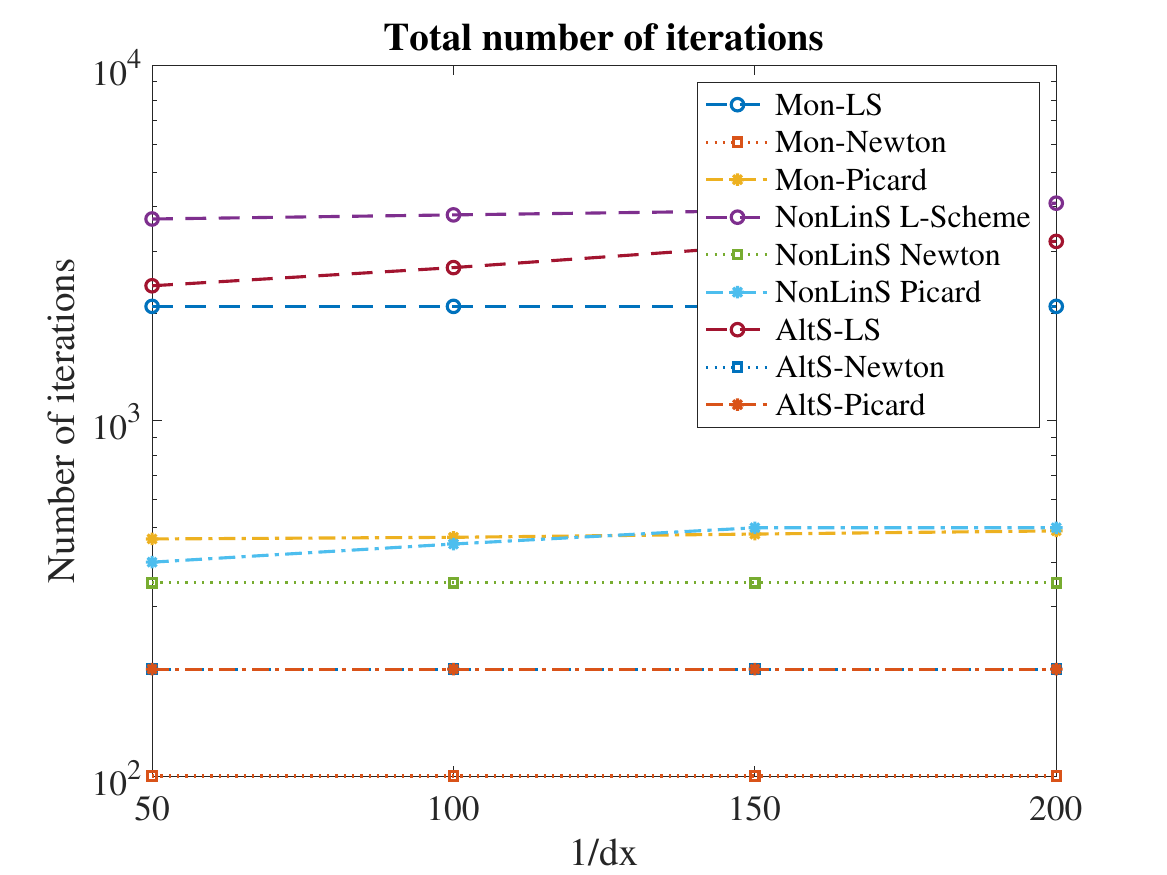}
  \caption{Example 2A: Logarithmic plot of numbers of iterations in unsaturated porous medium}
  \label{fig:iterationwell}
\end{figure}
%Saturated porous media fixed dt
\begin{table}[]
\centering
  \scalebox{.6}{
  \begin{tabular}{ c | c c | c c  c c  | c c c c}
\hline \hline
& Monolithic & &  & NonLinS &  & & & AltLinS &&\\
\hline
& Newton & & & Newton &  & & &Newton  & &\\
\hline
&  & & &  &  cond. \# & & &  & cond. \# &\\
\hline
dx &\# iterations & condition \#  & \# iterations & Richards && Transport  & \# iterations & Richards && Transport\\ [0.5ex]
 1/50 & 100 & 7.6243e+09   & 350 & 22.2624    && 3.3828e+09 & 200 & 22.1804 &&3.1779e+09   \\ 
 1/100 & 100 & 4.2369e+10   & 350 & 25.9419 && 5.2695e+10 & 200 & 31.8223  & & 5.5839e+10  \\
 1/150 & 100 &  1.0394e+11  & 350  & 42.8667 &&  2.8324e+11& 200 & 33.0562  && 2.8324e+11\\
 1/200 & 100 & 1.9614e+11   & 350 &   56.4999 & & 9.1099e+11 & 200 & 42.8581  &&  9.1662e+11\\
 \hline \hline
 & L Scheme & & & L Scheme &  & & & L Scheme & & \\
 \hline
&  & & &  &  cond. \# & & &  & cond. \# &\\
\hline
dx &\# iterations & condition \#  & \# iterations & Richards && Transport  & \# iterations & Richards && Transport\\ [0.5ex]
1/50 & 2100 & 2.7971e+09      & 3700   & 4.2830  &&   1.6426e+09   &  2400 & 4.2489 &&  1.4735e+09  \\ 
 1/100 & 2100 &  2.5556e+10  & 3800    &   5.0706&&  2.6387e+10   &   2700 & 5.0005  && 2.6331e+10 \\
 1/150 & 2100 & 7.6270e+10   & 3900   &  6.3840& &   1.4791e+11	  &  3100  &6.3731 &  &1.4165e+11  \\
 1/200 & 2100 & 1.3652e+11  & 4100  & 8.1437  && 4.5556e+11  &  3200  &  8.2452  && 4.5159e+11 \\
 \hline\hline
 & Picard & & & Picard   & & & & Picard & &\\
 \hline
&  & & &  &  cond. \# & & &  & cond. \# &\\
\hline
$\tau$ &\# iterations & condition \#  & \# iterations & Richards && Transport  & \# iterations & Richards && Transport\\ [0.5ex]
 1/50 & 465  & 4.3612e+09 & 400 &  22.3247 && 1.7215e+09 & 200 & 22.1612 && 7.6372e+08  \\
 1/100 &470 & 3.2796e+10  & 450  & 25.9427 && 2.6395e+10 &  200 & 26.0764 && 1.3243e+10\\
 1/150 &  480 & 9.0401e+10 & 500 & 33.4899 && 1.4173e+11 &  200&  33.0714& &  7.0961e+10\\
 1/200 &  490 &   1.7576e+11  &500 & 42.7698&& 4.5570e+11 &  200 & 42.737 &&  2.2773e+11 \\
 \hline
\end{tabular}
}
\caption{Example 2A: unsaturated medium, fixed $\tau=1/50$}
\label{tab:2.0dx}
\end{table}
%Variably saturated porous media fixed dx=1/40
\begin{table}[]
\centering
  \scalebox{.6}{
  \begin{tabular}{ c | c c | c c  c c  | c c c c}
\hline \hline
& Monolithic & &  & NonLinS &  & & & AltLinS &&\\
\hline
& Newton & & & Newton &  & & &Newton  & &\\
\hline
&  & & &  &  cond. \# & & &  & cond. \# &\\
\hline
$\tau$ &\# iterations & condition \#  & \# iterations & Richards && Transport  & \# iterations & Richards && Transport\\ [0.5ex]
1/50 & 100 & 7.6243e+09 & 300 & 22.2624  && 3.3828e+09 & 200 & 22.1804 && 3.1779e+09 \\ 
 1/100 & 200 & 2.7797e+09  & 350 & 23.0671 && 8.0658e+08  & 400 &   22.9854&& 3.8663e+08 \\
 1/150 & 300 &  1.4883e+09 & 400  & 21.3871 & & 3.8252e+08 & 600 & 21.9674  & & 2.0727e+08 \\
 1/200 & 400 & 9.2045e+08 & 500 & 21.2462  && 1.9030e+08& 800 &  21.3715 & & 1.8499e+08 \\
 \hline \hline
 & L Scheme & & & L Scheme &  & & & L Scheme & & \\
 \hline
&  & & &  &  cond. \# & & &  & cond. \# &\\
\hline
$\tau$ &\# iterations & condition \#  & \# iterations & Richards && Transport  & \# iterations & Richards && Transport\\ [0.5ex]
 1/50 & 2100 & 2.7971e+09 & 3700 & 4.2830 &&1.6426e+09 &2400 &4.2489& &1.4735e+09 \\ 
 1/100 & 4100&8.0703e+08 & 7700 &4.1966 && 3.6673e+08 & 4200& 4.1408 && 1.8802e+08 \\
 1/150 &5900 & 3.8356e+08 &11250 &  4.1439 &&1.6219e+08 & 5800 & 4.1052 & & 8.4601e+07\\
 1/200 & 7700& 2.2641e+08 &14600& 4.0893&&9.1518e+07  &7200  & 4.0870 && 4.7818e+07\\
 \hline\hline
 & Picard & & & Picard   & & & & Picard & &\\
 \hline
&  & & &  &  cond. \# & & &  & cond. \# &\\
\hline
$\tau$&\# iterations & condition \#  & \# iterations & Richards && Transport  & \# iterations & Richards && Transport\\ [0.5ex]
 1/50 & 465 & 4.3612e+09 & 400 &22.3247 &&1.7215e+09 &200&22.1612 &&7.6372e+08 \\
 1/100 & 880 &  1.3673e+09 & 600 &  23.0000&&1.9895e+08 & 400  & 21.5095 && 1.9878e+08  \\
 1/150 &1300 & 6.6937e+08 & 900 &22.4168 &&9.6242e+07 &600 & 21.2971 &&8.8586e+07  \\
 1/200 &1700& 4.0868e+08 &1200 &21.2490 && 5.0376e+07&750 & 21.2490&& 5.0376e+07\\
 \hline
\end{tabular}
}
\caption{Example 2A: unsaturated medium, fixed dx=1/50}
\label{tab:2.0dt}
\end{table}
\subsection*{Example 2B: well in variably saturated porous media}
\paragraph{}
This fourth numerical example is obtained by changing the initial condition for pressure in the example 2A. We use the same $p^0$ as in example 1B. The profiles of pressure and concentration at the beginning and end of the simulation, i.e. at $T=1$ hour, are presented in Fig. \ref{itwellvarsat}. We can observe smaller changes, compared to the previous example, due to a smaller time interval (1 hour versus 1 day). 
\begin{figure}[]
\begin{center}
\begin{subfigure}{.45\textwidth}  
  \includegraphics[width=1\linewidth]{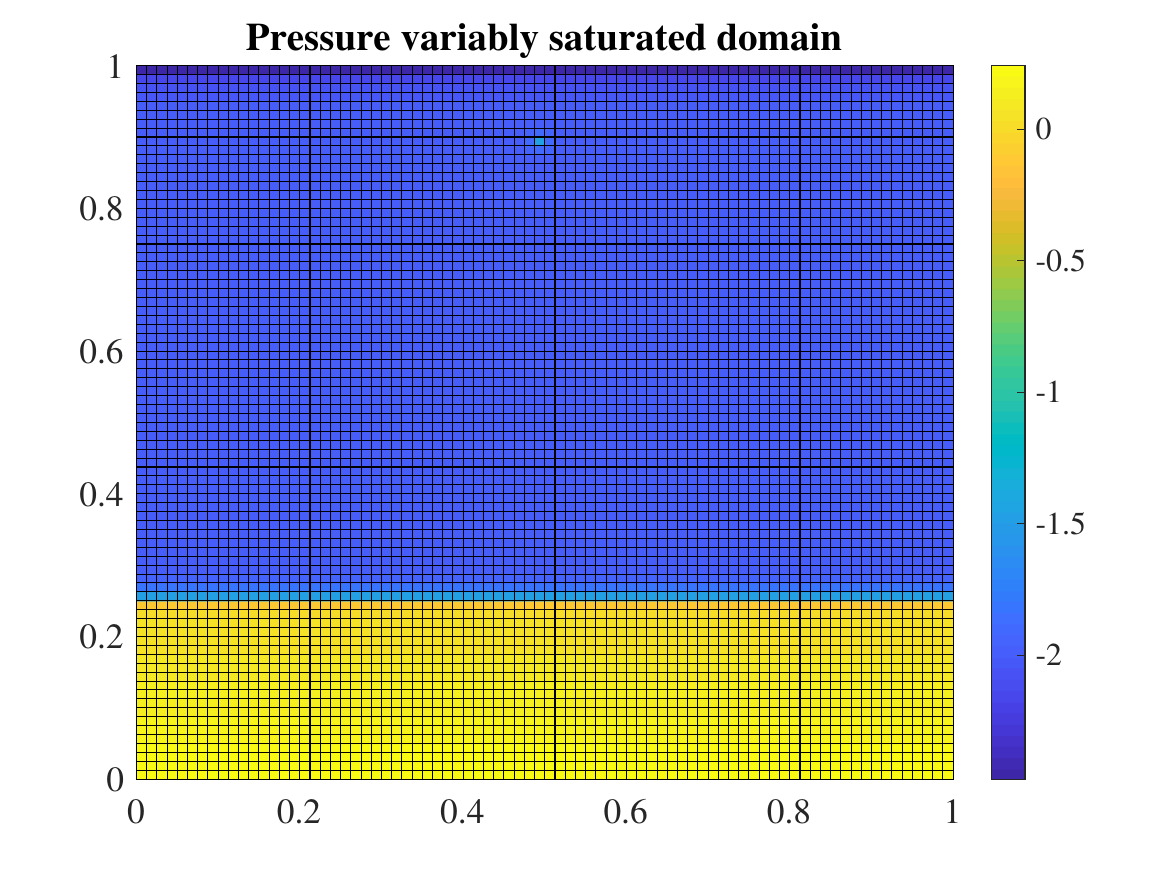}
  \caption{Pressure profile at first time step}
  \label{fig:variablypressure}
\end{subfigure}
  \begin{subfigure}{.45\textwidth}
  \includegraphics[width=1\linewidth]{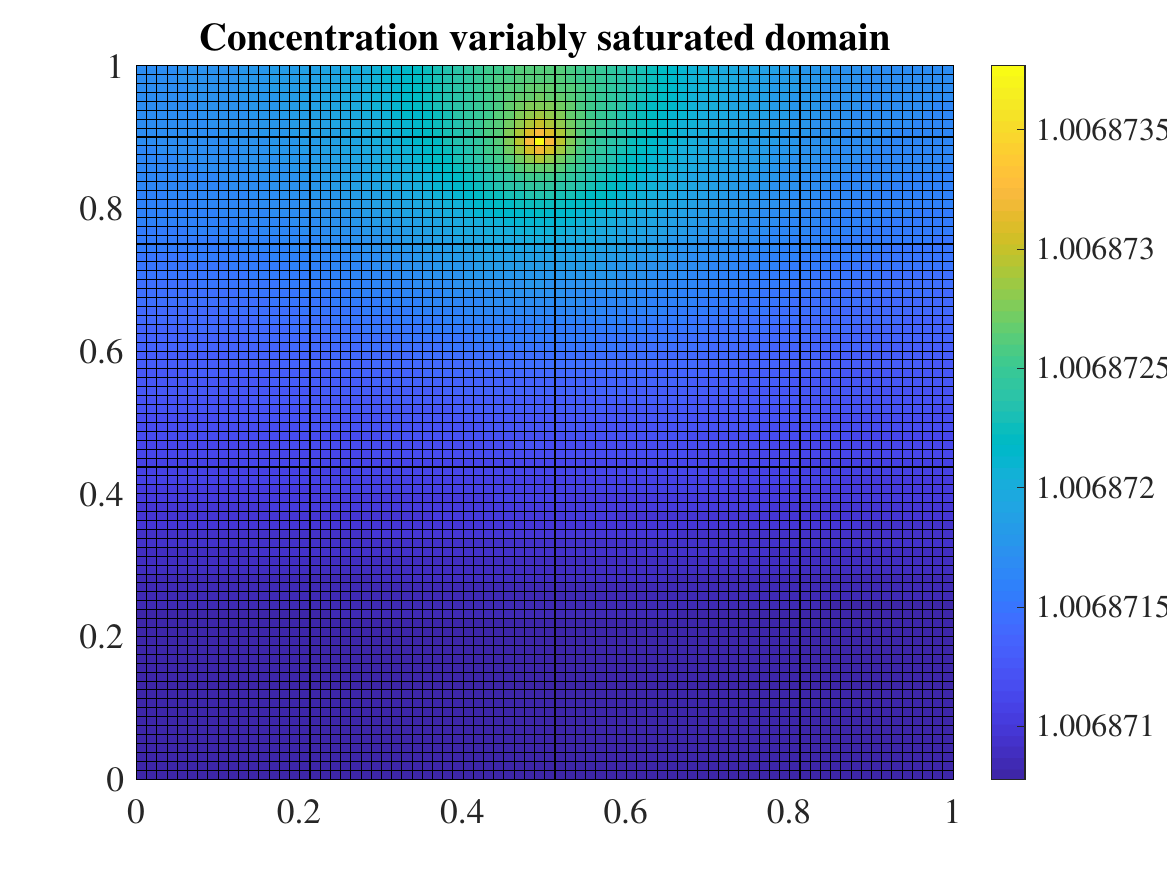}
  \caption{Concentration profile at first time step}
  \label{fig:variablyconcentration}
\end{subfigure}
\caption{Example 2B: plots of pressure and concentration in unsaturated medium, the simulations were done with $dx = 1/80$ and $\tau =dx/100$}
\end{center}
\end{figure}
\begin{figure}[]
\begin{center}
  \begin{subfigure}{.45\textwidth}  
  \includegraphics[width=1\linewidth]{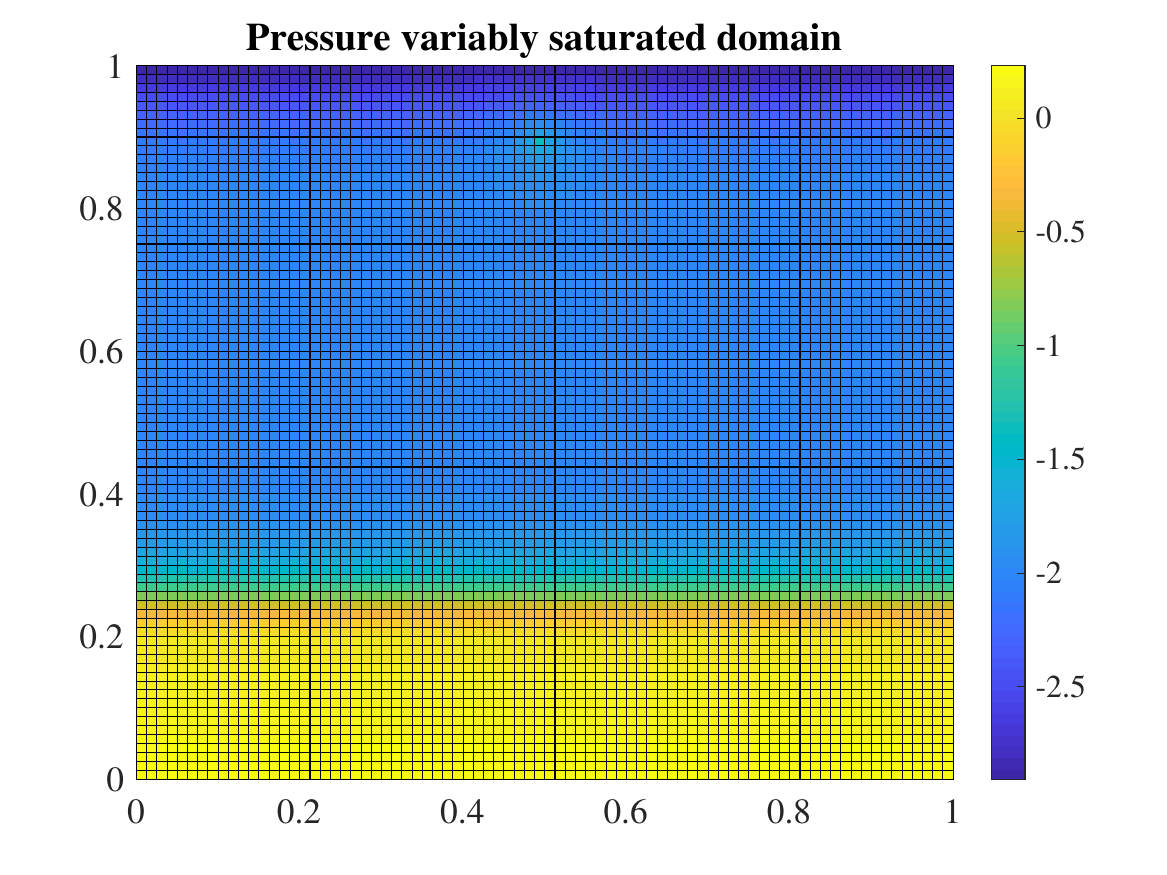}
  \caption{Pressure profile after one hour}
  \label{fig:variablypressure}
\end{subfigure}
  \begin{subfigure}{.45\textwidth}
  \includegraphics[width=1\linewidth]{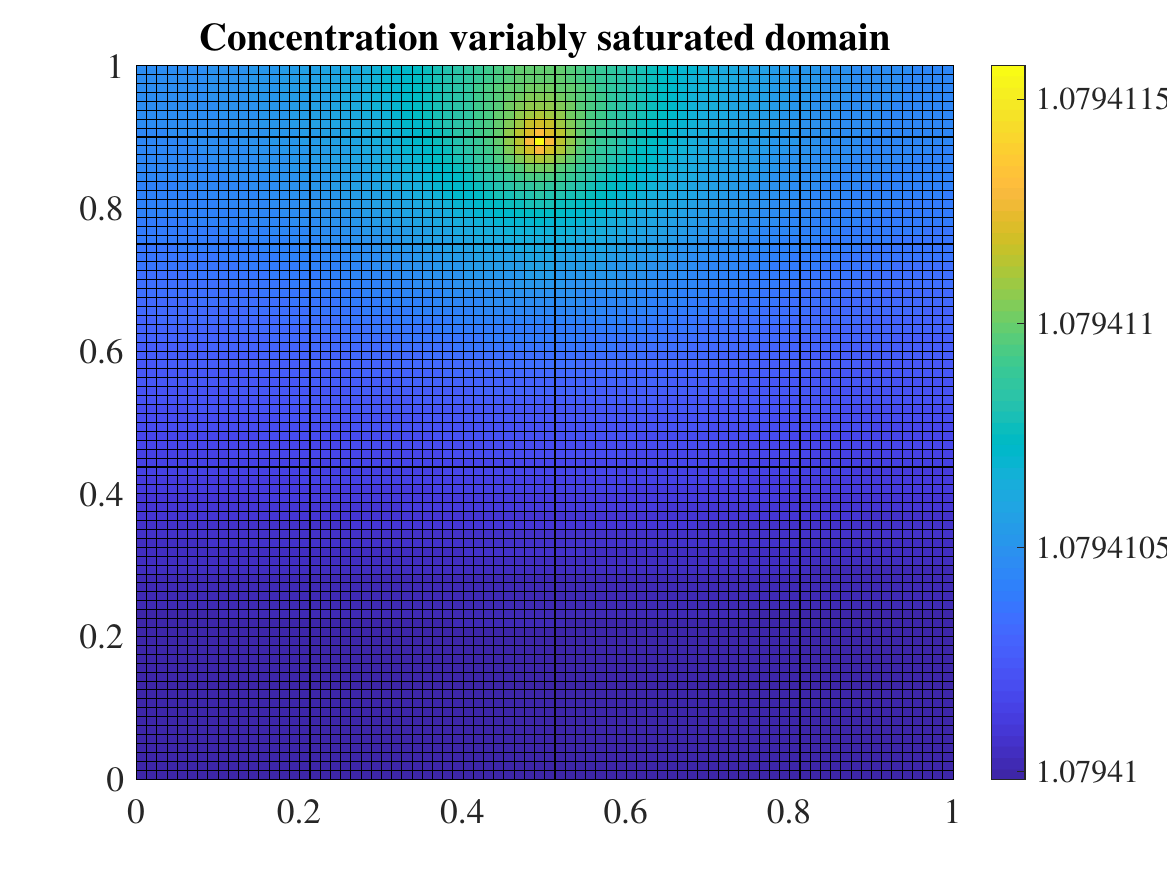}
  \caption{Concentration profile after one hour}
  \label{fig:variablyconcentration}
\end{subfigure}
\caption{Example 2B: pressure and concentration profiles after one hour. The simulations were done with $dx = 1/80$ and $\tau =dx/100$}
\label{fig:wellvarsat}
\end{center}
\end{figure}
In Fig. \ref{itwellvarsat} we present the total number of iterations for the different schemes applied to example 2B. Similar to the example 1B, due to the degeneracy of the Richards equation, many of the considered schemes show convergence problems. In the Tables \ref{tab:2.1dx},\ref{tab:2.1dt} we study the convergence of the schemes and the condition number of the associated linear systems. The results are very similar with the previous examples, with the L-scheme based solvers being the most robust one for all the cases and with the alternate method being faster than the classical splitting schemes.
\begin{figure}[]
\centering
  \includegraphics[scale =.4]{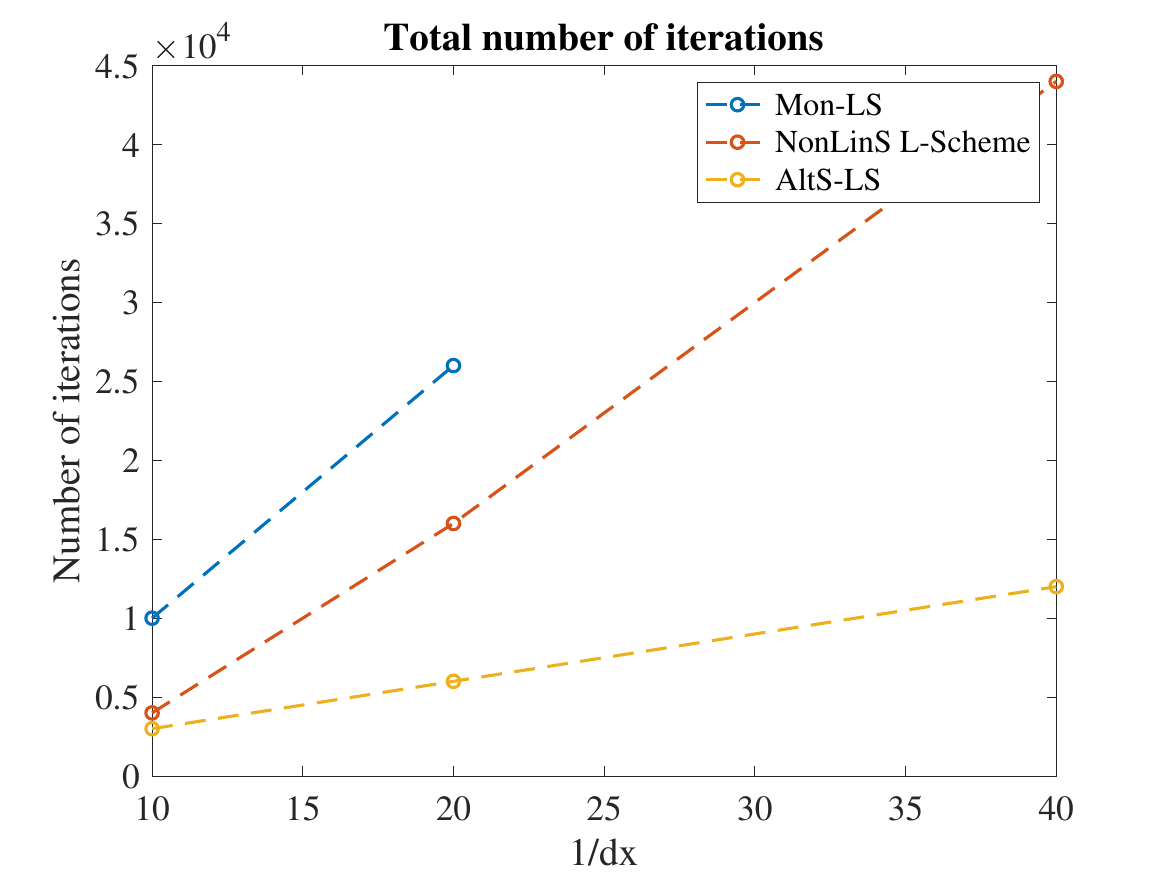}
  \caption{Example 2B: total number of iterations for different algorithms}
  \label{itwellvarsat}
\end{figure}

\begin{table}[]
\centering
  \scalebox{.6}{
  \begin{tabular}{ c | c c | c c  c c  | c c c c}
\hline \hline
& Monolithic & &  & NonLinS &  & & & AltLinS &&\\
\hline
& Newton & & & Newton &  & & &Newton  & &\\
\hline
&  & & &  &  cond. \# & & &  & cond. \# &\\
\hline
dx &\# iterations & condition \#  & \# iterations & Richards && Transport  & \# iterations & Richards && Transport\\ [0.5ex]
 1/10 & - & 2.2435e+10  & - & 20.9641 && 3.2267e+08  & - & 20.9795 && 3.2267e+08 \\ 
 1/20 & - & 1.2309e+11   & - &  20.9642  && 5.3778e+08 & - & 20.9800 & &  5.3778e+08 \\
 1/40 & - & 4.5128e+11    & -  & 20.9644   && 1.2100e+09 & - & 20.9826  && 1.2100e+09  \\
 \hline \hline
 & L Scheme & & & L Scheme &  & & & L Scheme & & \\
 \hline
&  & & &  &  cond. \# & & &  & cond. \# &\\
\hline
dx &\# iterations & condition \#  & \# iterations & Richards && Transport  & \# iterations & Richards && Transport\\ [0.5ex]
1/10 & 10000 & 1.1768e+04  & 4000 & 5.000   &&  6.2544e+03 &  3000  & 5.000   && 6.2544e+03 \\ 
 1/20 & 26000 & 2.5237e+04 & 16000 &  5.000  && 1.4067e+04 &  6000 &   5.000   && 1.4067e+04  \\
 1/40 & -         & 5.1895e+04 & 44000 &  5.000 & &  2.8441e+04   & 12000  & 5.000 &  & 2.8440e+04 \\
 \hline\hline
 & Picard & & & Picard   & & & & Picard & &\\
 \hline
&  & & &  &  cond. \# & & &  & cond. \# &\\
\hline
dx &\# iterations & condition \#  & \# iterations & Richards && Transport  & \# iterations & Richards && Transport\\ [0.5ex]
 1/10 & - &4.4870e+10 &  -  &  20.9641 && 3.2267e+08 & - & 20.9726 && 3.2267e+08 \\
 1/20 & - & 2.4617e+11 &  - & 20.9642 &&   5.3778e+08& - & 20.9732 && 5.3778e+08\\
 1/40 & - & 9.0255e+11 & - & 20.9644&&  1.2100e+09 &  - &20.9752 & & 1.2100e+09\\
 \hline
\end{tabular}
}
\caption{Example 2B: variably saturated medium, fixed $\tau$= dx/100}
\label{tab:2.1dx}
\end{table}
\begin{table}[]
\centering
  \scalebox{.6}{
  \begin{tabular}{ c | c c | c c  c c  | c c c c}
\hline \hline
& Monolithic & &  & NonLinS &  & & & AltLinS &&\\
\hline
& Newton & & & Newton &  & & &Newton  & &\\
\hline
&  & & &  &  cond. \# & & &  & cond. \# &\\
\hline
$\tau$&\# iterations & condition \#  & \# iterations & Richards && Transport  & \# iterations & Richards && Transport\\ [0.5ex]
 1/1000 & - & 2.2435e+10& - & 20.9641 && 3.2267e+08 & - & 20.9795 && 3.2267e+08\\ 
 1/2000 & - &  2.2444e+10& - & 20.9641 && 6.4533e+08 & - & 20.9799 &&  6.4533e+08\\
 1/4000 & - &  2.2461e+10& 20000  & 20.9640 &&  1.2907e+09& - & 20.9807 && 1.2907e+09\\
 \hline \hline
 & L Scheme & & & L Scheme &  & & & L Scheme & & \\
 \hline
&  & & &  &  cond. \# & & &  & cond. \# &\\
\hline
$\tau$&\# iterations & condition \#  & \# iterations & Richards && Transport  & \# iterations & Richards && Transport\\ [0.5ex]
 1/1000 & 10000 & 1.1768e+04 & 4000 & 5.000 && 6.2544e+03 & 3000& 5.000   &&6.2544e+03 \\ 
 1/2000 & 14000 & 5.9591e+03 & 6000  & 5.000  &&  3.2036e+03 &  6000&5.000  && 3.2036e+03 \\
 1/4000 & 20000 & 3.0483e+03 &12000 & 5.000  && 1.6697e+03&  12000&5.000 &&  1.6697e+03 \\
 \hline\hline
 & Picard & & & Picard   & & & & Picard & &\\
 \hline
&  & & &  &  cond. \# & & &  & cond. \# &\\
\hline
$\tau$&\# iterations & condition \#  & \# iterations & Richards && Transport  & \# iterations & Richards && Transport\\ [0.5ex]
 1/1000 & - & 4.4870e+10& - & 20.9641 && 3.2267e+08 & - &20.9726 && 3.2267e+08 \\ 
 1/2000 & - & 2.4617e+11 & - & 20.9642&& 6.4533e+08 & - & 20.9731 &&  6.4533e+08\\
 1/4000 & - &  9.0255e+11& - & 20.9644 &&  1.2907e+09& - & 20.9739  &&1.2907e+09\\
 \hline
\end{tabular}
}
\caption{Example 2B: variably saturated medium, fixed dx=1/10}
\label{tab:2.1dt}
\end{table}

\subsection{\textcolor{blue}{Example 3: Highly heterogeneous porous medium}}
\paragraph*{}
\textcolor{blue}{
Physical porous media, such underground reservoirs, are characterized by highly heterogeneous properties. Here we investigate a domain $\Omega \subset \mathbb{R}^2$ with highly heterogeneous porosity and permeability as presented in Fig. \ref{fig:hetereogeneousrock}.
}

\textcolor{blue}{
We consider the problem already studied in Example 1B, using the same initial conditions and parameters. We set the external forces to be equal to zero so that the flow and transport are governed only by the initial and boundary conditions. The problem investigated is highly heterogeneous and, thanks to the initial pressure $p^0$, also variably saturated, presenting a discontinuity in the water content. Due to the low conductivity, Fig. \ref{fig:conductivity}, a large time interval is considered, $T=10^5$. 
}

\textcolor{blue}{
We can observe, in Tables \ref{tab:3dx} and \ref{tab:3dt} and in Fig. \ref{itheterogeneous}, the total numbers of iterations required by each solving algorithm and the condition numbers of the associated linearized systems. For this particular problem, it is interesting to notice that, even tough we are investigating a large time domain, the L-schemes converges for a time step $\tau=T/10$. Considering the results presented in the previous examples, where for smaller time steps all the schemes converged, we tested also $\tau = [T/10^2,T/10^3,T/10^4]$. Anyhow, due to the highly complex domain, both Newton method and modified Picard did not converge. We did not test smaller time steps because the resulting number of iterations would have been much larger than the one obtained with the original $\tau$ and the L-scheme.
}

\begin{figure}[]
\begin{center}
  \begin{subfigure}{.49\textwidth}  
  \includegraphics[width=1\linewidth]{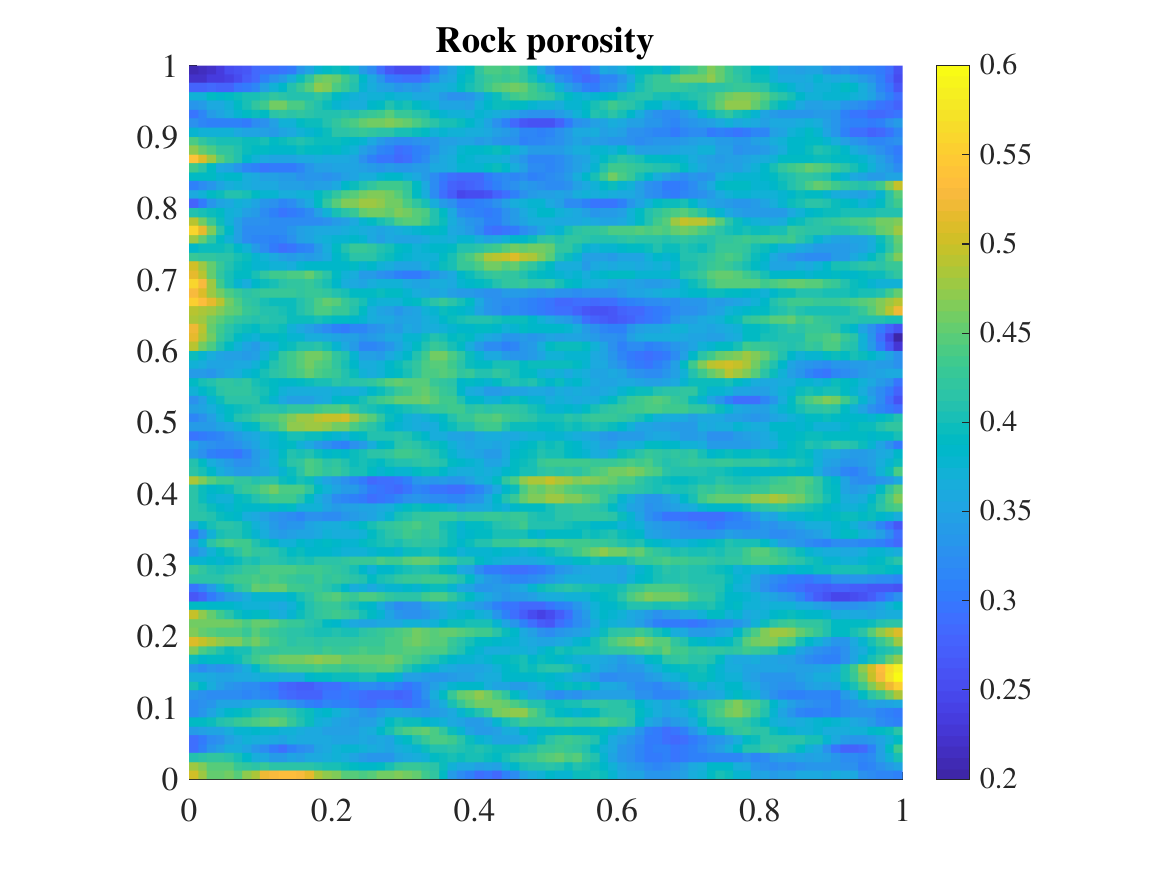}
  \caption{Porosity of the domain $\Omega$}
  \label{fig:porosity}
\end{subfigure}
  \begin{subfigure}{.49\textwidth}
  \includegraphics[width=1\linewidth]{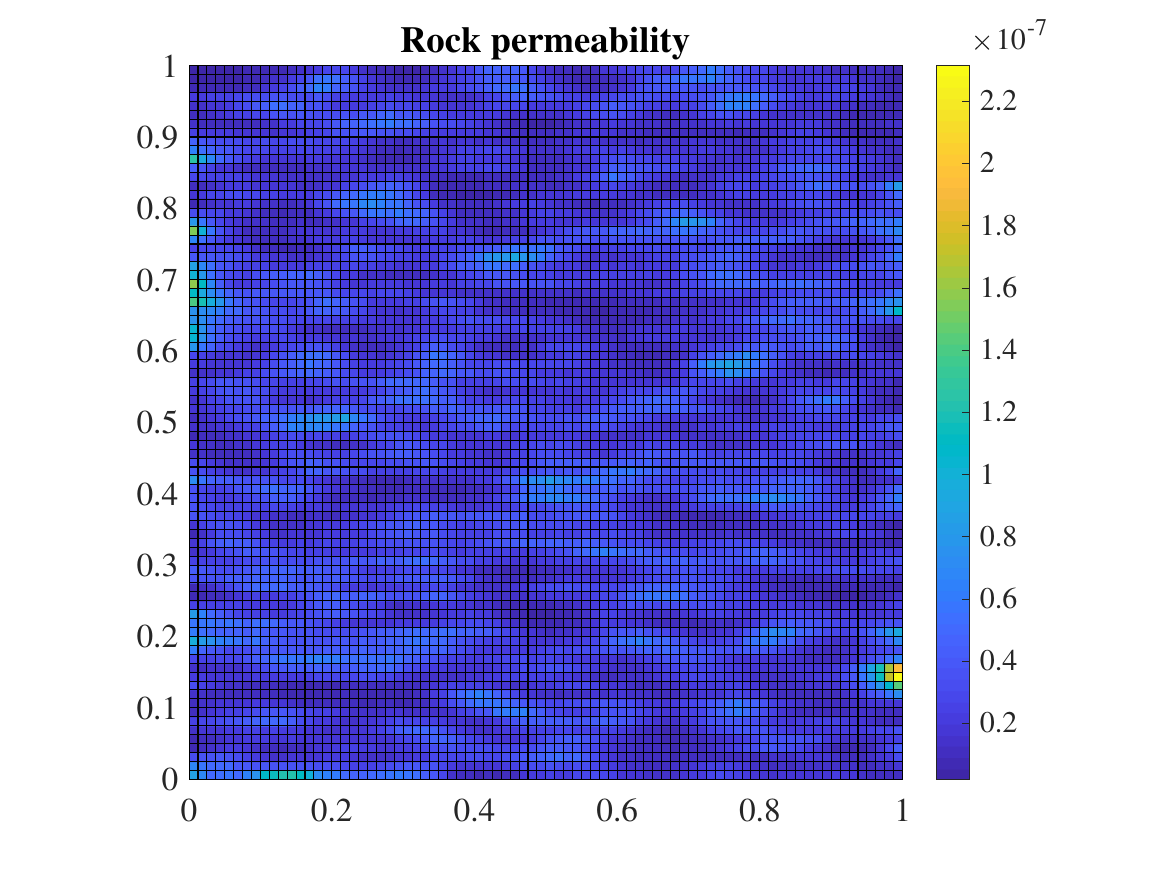}
  \caption{Permeability of the domain $\Omega$}
  \label{fig:conductivity}
\end{subfigure}
\caption{Example 3: highly heterogeneous domain}
\label{fig:hetereogeneousrock}
\end{center}
\end{figure}

\begin{figure}[]
\centering
  \includegraphics[scale =.4]{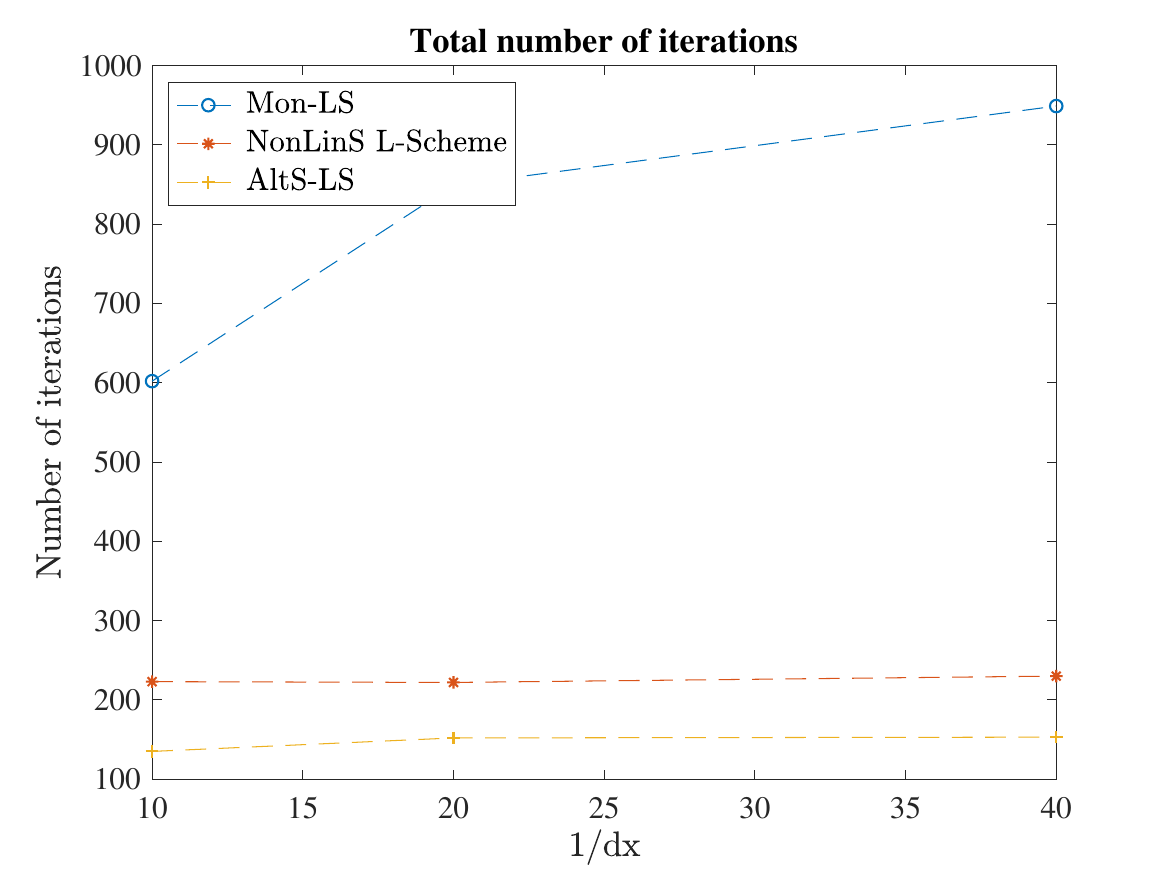}
  \caption{Example 3: total number of iterations for different algorithms}
  \label{itheterogeneous}
\end{figure}

\begin{table}[]
\centering
  \scalebox{.6}{
  \begin{tabular}{ c | c c | c c  c c  | c c c c}
\hline \hline
& Monolithic & &  & NonLinS &  & & & AltLinS &&\\
\hline
& Newton & & & Newton &  & & &Newton  & &\\
\hline
&  & & &  &  cond. \# & & &  & cond. \# &\\
\hline
dx &\# iterations & condition \#  & \# iterations & Richards && Transport  & \# iterations & Richards && Transport\\ [0.5ex]
 1/10 & - & 1.7934e+09  & - & 240.1589 && 6.6555e+07  & - &240.1589  &&  6.6555e+07\\ 
 1/20 & - &  3.1287e+09  & - &  1.2146e+03  && 2.6668e+08 & - &1.2146e+03  & & 2.6668e+08  \\
 1/40 & - & 4.7011e+10  & -  &  3.3058e+03  && 1.0721e+09 & - &3.3058e+03   && 1.0721e+09  \\
 \hline \hline
 & L Scheme & & & L Scheme &  & & & L Scheme & & \\
 \hline
&  & & &  &  cond. \# & & &  & cond. \# &\\
\hline
dx &\# iterations & condition \#  & \# iterations & Richards && Transport  & \# iterations & Richards && Transport\\ [0.5ex]
1/10 &602  & 3.0502e+07 & 223 & 13.3371 && 1.8192e+07  &  135& 13.3371 && 1.8192e+07 \\ 
 1/20 & 849 & 1.0614e+08  & 222 & 47.4216   &&  6.3261e+07 &152  &  47.4216   && 6.3261e+07  \\
 1/40 & 949  & 4.2918e+08 &230  &224.9096   & & 2.9895e+08   & 153  & 224.9096 &&  2.9895e+08\\
 \hline\hline
 & Picard & & & Picard   & & & & Picard & &\\
 \hline
&  & & &  &  cond. \# & & &  & cond. \# &\\
\hline
dx &\# iterations & condition \#  & \# iterations & Richards && Transport  & \# iterations & Richards && Transport\\ [0.5ex]
 1/10 & - & 1.5365e+09 &  -  & 308.0140  && 2.1527e+07 & - & 308.0140 && 2.1527e+07 \\
 1/20 & - & 5.4220e+09 &  - & 946.9821  &&  9.3559e+07 & - & 946.9821 &&9.3559e+07  \\
 1/40 & - & 4.0185e+10  & - & 4.3966e+03  && 3.5695e+08  &  - &4.3966e+03 & & 3.5695e+08\\
 \hline
\end{tabular}
}
\caption{Example 3: heterogeneous variably saturated medium, fixed $\tau$= T/10}
\label{tab:3dx}
\end{table}

\begin{table}[]
\centering
  \scalebox{.6}{
  \begin{tabular}{ c | c c | c c  c c  | c c c c}
\hline \hline
& Monolithic & &  & NonLinS &  & & & AltLinS &&\\
\hline
& Newton & & & Newton &  & & &Newton  & &\\
\hline
&  & & &  &  cond. \# & & &  & cond. \# &\\
\hline
$\tau$ &\# iterations & condition \#  & \# iterations & Richards && Transport  & \# iterations & Richards && Transport\\ [0.5ex]
 T/$10$ & - & 1.7934e+09   & - &240.1589 && 6.6555e+07  & - &240.1589  &&  6.6555e+07\\ 
 T/$10^2$ & - & 6.7021e+08   & - &  483.2030  && 6.6082e+06  & - & 483.2030 & & 6.6082e+06  \\
 T/$10^3$ & - & 4.6611e+08 & -  &  3.4913e+03  && 6.6008e+05 & - & 3.4913e+03  && 6.6008e+05  \\
 \hline \hline
 & L Scheme & & & L Scheme &  & & & L Scheme & & \\
 \hline
&  & & &  &  cond. \# & & &  & cond. \# &\\
\hline
$\tau$ &\# iterations & condition \#  & \# iterations & Richards && Transport  & \# iterations & Richards && Transport\\ [0.5ex]
T/$10$ &602  & 3.0502e+07 & 223 & 13.3371 && 1.8192e+07  &  135& 13.3371 && 1.8192e+07 \\ 
 T/$10^2$ &6880  & 3.0503e+06  &2350 &   2.3236 && 1.6678e+06  & 1431  &   2.3236  && 1.6678e+06 \\
 T/$10^3$ & 23188  & 3.0528e+05 & 26890 & 1.5385  & &  1.5358e+05  & 7950  &1.5385  && 1.5358e+05 \\
 \hline\hline
 & Picard & & & Picard   & & & & Picard & &\\
 \hline
&  & & &  &  cond. \# & & &  & cond. \# &\\
\hline
$\tau$ &\# iterations & condition \#  & \# iterations & Richards && Transport  & \# iterations & Richards && Transport\\ [0.5ex]
 T/$10$ & - & 1.5365e+09 &  -  & 308.0140  && 2.1527e+07 & - & 308.0140 && 2.1527e+07 \\
 T/$10^2$ & - & 8.1784e+09 &  - & 341.0311  && 2.5362e+06  & - &341.0311  &&2.5362e+06 \\
 T/$10^3$ & - & 1.0572e+10   & - & 366.8343  && 2.8550e+05  &  - & 366.8343& &2.8550e+05 \\
 \hline
\end{tabular}
}
\caption{Example 3: heterogeneous variably saturated medium, fixed dx= 1/10}
\label{tab:3dt}
\end{table}

\newpage
\section{Conclusions}\label{conclusion}

%We observe as the proposed alternate splitting approach is a valid alternative to the common solving formulation, it requires fewer iterations and it gives equally accurate results. We compare also the different linearization schemes concluding that, only the $L$-scheme converges in case of a degenerative Richards equation, due to a fast diffusion process ($\theta'= 0$).

\paragraph{}
In this paper we considered surfactant transport in variably saturated porous media. The water flow and the transport are in this case fully coupled. Three linearization techniques were considered: the Newton method, the modified Picard and the L-scheme. Based on these, monolithic and splitting schemes were designed, analyzed and numerically tested. We conclude that the only quadratic convergent scheme is the monolithic Newton, that the L-scheme based solvers are the most robust ones and produce well-conditioned linear systems and that the alternative schemes are faster than the classical splitting approaches.

\textcolor{blue}{
Although we recognized the existence of improved Newton solvers, e.g. \cite{jenny2009, lee2016,wang2013,younis2010}, we believe that, also due to its extreme simplicity, the L-scheme is a powerful tool and can be particularly useful in degenerative cases here investigated.
}

\begin{acknowledgments}
The research of D. Illiano was funded by VISTA, a collaboration between the Norwegian Academy of Science and Letters and Equinor, project number 6367, project name: adaptive model and solver simulation of enhanced oil recovery. The research
of I.S. Pop was supported by the Research Foundation-Flanders (FWO), Belgium through the Odysseus programme (project G0G1316N) and Equinor through the Akademia grant. 
\\We thank our colleagues from the \emph{Sintef} research group. In particular Olav Moyner PhD, for assistance with the implementation of the numerical examples in \emph{MRST}, the toolbox based on Matlab developed at \emph{Sintef} itself.
\end{acknowledgments}

\bibliographystyle{spmpsci}      % mathematics and physical sciences

\begin{thebibliography}{00}
	
	\bibitem{aavatsmark}
	I. Aavatsmark,
	\emph{An introduction to multipoint flux approximations for quadrilateral grids, Computational Geosciences Volume 6, Issue 3-4, Pages 405-432},
	2002.
	
	\bibitem{agosti}
	A. Agosti, L. Formaggia, A. Scotti, 
	\emph{Analysis of a model for precipitation and dissolution coupled with a Darcy flux, J. Math. Anal. Appl. 431, Issue 2, Pages 752-781}, 2015.
	
	
	\bibitem{AltLuckhaus}
	W. Alt, H. Luckhaus,
	\emph{Quasilinear elliptic-parabolic differential equations, S. Math Z, Volume 183, Issue 3, Pages 311-341},
	1983.
	
	\bibitem{arbogast1996}
	T. Arbogast, M. F. Wheeler,
	\emph{A nonlinear mixed finite element method for a degenerate parabolic equation arising in flow in porous media, SIAM J. Numer. Anal., Volume 33, Issue 4, Pages 1669-1687},
	1996. 
	
	\bibitem{barrett1997}
	J. W. Barrett, P. Knabner,
	\emph{Finite element approximation of the transport of reactive solutes in porous media. Part 1: error estimates for nonequilibrium adsorption processes, SIAM J. Numer. Anal., Volume 34, Issue 1, Pages 201-227},
	1997.  
	
	\bibitem{bause2010}
	M. Bause, J. Hoffmann, P. Knabner,
	\emph{First-order convergence of multi-point flux approximation on triangular grids and comparison with mixed finite element methods, Numer. Math., Volume 116, Issue 1, Pages 1-29},
	2010.
	
	\bibitem{RichardsBerardi}
	M. Berardi, F. Difonzo, M. Vurro, L. Lopez,
	\emph{The 1D Richards’ equation in two layered soils: a Filippov approach to treat discontinuities, Adv Water Resour, Volume 115, Pages 264-272},
	2018.
	
	\bibitem{putti}
L. Bergamaschi, M. Putti,
\emph{Mixed finite elements and Newton-type linearizations for the solution of Richards' equation, Int. J. Num. Meth. Engng., Volume 45, Issue 8, Pages 1025-1046},
1999.

%	\bibitem{brunner2012}
%	F. Brunner, F. A. Radu, M. Bause, P. Knabner,
%	\emph{Optimal order convergence of a modified BDM1 mixed finite element scheme for reactive transport in porous media, Adv. Water Resour., Volume 35, Pages 163-171},
%	2008. 
	
	\bibitem{cances}
	C. Cances, I.S. Pop, M. Vohralik, 
	\emph{An a posteriori error estimate for vertex-centered finite volume discretizations of immiscible incompressible two-phase flow, Math. Comp. Vol. 83, Pages 153-188}, 2014.
	
	\bibitem{celia}
	M. Celia, E. Bouloutas, R. L. Zarba
	\emph{A General Mass-Conservative Numerical Solution for the Unsaturated Flow Equation, Adv Water Resour, Volume 26, Issue 7, Pages 1483-1496},
	1990.
	
	\bibitem{christofi}
	N. Christofi,  I. B. Ivshina,
	\emph{Microbial surfactants and their use in field studies of soil remediation, J. Appl. Microbiol., Volume 93, Issue 6, Pages 915-929},
	2002.
	
	\bibitem{dawson1998}
	C. Dawson,
	\emph{Analysis of an upwind-mixed finite element method for nonlinear contaminant transport equations, SIAM J. Numer. Anal., Volume 35, Issue 5, Pages 1709-1724},
	1998.  
	
	\bibitem{eymard1999}
	R. Eymard, M. Gutnic, D. Hilhorst,
	\emph{The finite volume method for Richards equation, Comput. Geosci., Volume 3, Issue 3-4, Pages 256-294}, 1999.  
	
	\bibitem{eymard2006}
	R. Eymard, D. Hilhorst, M. Vohral,
	\emph{A combined finite volume- nonconforming/mixed-hybrid finite element scheme for degenerate parabolic problems, Numer. Math., Volume 105, Issue 1, Pages 73-131},
	2006.    
	
\bibitem{Farthing}
M.W. Farthing, F.L. Ogden,
\emph{Numerical Solution of Richards’ Equation: A Review of Advances and Challenges, Soil Sci. Soc. Amer. J. Volume 81, Pages 1257-1269},
2017.

	\bibitem{Gallo}
	C. Gallo, G. Manzini,
	\emph{A mixed finite element/finite volume approach for solving biodegradation transport in groundwater, Int. J. Numer. Meth. Fl., Volume 26, Issue 5, Pages 533-556},
	1998.  
	
	\bibitem{vanG}
	M. van Genuchten,
	\emph{A Closed-form Equation for Predicting the Hydraulic Conductivity of Unsaturated Soils, Soil Sci. Soc. Am. J. Volume 44, Issue 5, Pages 892-898},
	1980. 
	
	\bibitem{BookHelmig}
	R. Helmig,
	\emph{Multiphase flow and transport processes in the subsurface: a contribution to the modeling of hydrosystems, Springer-Verlag},
	1997.
	
	\bibitem{surfactant1999}
	E. J. Henry, J. E. Smith, A. W.  Warrick,
	\emph{Solubility effects on surfactant-induced unsaturated flow through porous media,
		Journal of Hydrology, Volume 223, Issues 3–4, Pages 164-174},
	1999. 
	
%	\bibitem{hou1997}
%	T. Y. Hou, X. H. Wu,
%	\emph{A multiscale finite element method for elliptic problems in composite materials and porous media, J. Comput. Phys., Volume 134, Issue 1, Pages 169-189},
%	1997.
	
	\bibitem{salt}
	D. Husseini,
	\emph{Effects of Anions acids on Surface Tension of Water, Undergraduate Research at JMU Scholarly Commons},
	2015. 
	
\bibitem{jenny2009}
P. Jenny, H. A. Tchelepi, S. H. Lee,
\emph{Unconditionally convergent nonlinear solver for hyperbolic conservation laws with S-shaped flux functions, J Comput Phys, Volume 228, Issue 20,  Pages 7497-7512},
2009	
	
	\bibitem{Karagunduz}
	A. Karagunduz, M. H. Young,  K. D. Pennell,
	\emph{Influence of surfactants on unsaturated water flow and solute transport, Water Resour. Res., Volume 51, Issue 4, Pages 1977-1988},
	2015. 
	
	\bibitem{klausenMPFA}
	R. A. Klausen, F. A. Radu, G. T. Eigestad,
	\emph{Convergence of MPFA on triangulations and for Richards' equation, Int. J. Numer. Meth. Fl., Volume 58, Issue 12, Pages 1327-1351},
	2008. 
	
	\bibitem{surf}
	P. Knabner, S. Bitterlich, R. I. Teran, A. Prechtel, E. Schneid,
	\emph{Influence of Surfactants on Spreading of Contaminants and Soil Remediation, Springer},
	2003.
	
	\bibitem{kumar2013}
	K. Kumar, I. S. Pop, F. A. Radu, 
	\emph{Convergence analysis of mixed  numerical schemes for reactive flow in a porous medium,  SIAM J. Numer. Anal., Volume 51, Issue 4, Pages 2283-2308},
	2013.  
	
\bibitem{lee2016}
S. H. Lee, Y. Efendiev,
\emph{C1-Continuous relative permeability and hybrid upwind discretization of three phase flow in porous media, Adv in Water Resour, Volume 96, Pages 209-224},
2016.
	
	\bibitem{mrst}
	K.-A. Lie,
	\emph{ An Introduction to Reservoir Simulation Using MATLAB: User guide for the Matlab Reservoir Simulation Toolbox (MRST), SINTEF ICT},
	2016
	
	\bibitem{List2016}
	F. List, F. A. Radu,
	\emph{A study on iterative methods for solving Richards' equation, Comput. Geosci., Volume 20, Issue 2, Pages 341-353},
	2016.
	
	 \bibitem{Mitra}
	K. Mitra, I. S. Pop, 
	\emph{A modified L-Scheme to solve nonlinear diffusion problems, Comput. Math. Appl. Volume 77, Pages 1722-1738},
	2019. 
	
	\bibitem{Nochetto}
	R. Nochetto, C. Verdi, 
	\emph{Approximation of degenerate parabolic problems using numerical integration. SIAM J. Numer. Anal. 25, Pages 784-814}, 1988.
	
	\bibitem{prechtel}
	A. Prechtel, P. Knabner,
	\emph{Accurate and efficient simulation of coupled water flow and nonlinear reactive transport in the saturated and vadose zone - application to surfactant enhanced and intrinsic bioremediation, Int. J. Water Resour. D.,Volume 47, Pages 687-694},
	2002.
	
	\bibitem{Pop2004}
	I. S. Pop, F. A. Radu, P. Knabner,
	\emph{Mixed finite elements for the  Richards' equation: linearization procedure, J. Comput. Appl. Math., Volume 168, Issue 1, Pages 365-373},
	2004.
	
	
	\bibitem{radu2010}
	F. A. Radu, I. S. Pop, S. Attinger,
	\emph{Analysis of an Euler implicit, mixed finite element scheme for reactive solute transport in porous media, Numer. Methods Partial. Differ. Equ., Volume 26, Issue 2, Pages 320-344},
	2010.
	
	\bibitem{radu2004}
	F. A. Radu, I. S. Pop, P. Knabner,
	\emph{Order of convergence estimates for an Euler implicit, mixed finite element discretization of Richards' equation, SIAM J. Numer. Anal., Volume 42, Issue 4, Pages 1452-1478},
	2004.   
	
	\bibitem{pop2}
	F. A. Radu, I. S. Pop, P. Knabner,
	\emph{On the convergence of the Newton method for the mixed finite element discretization of a class of degenerate parabolic equation, Numerical Mathematics and Advanced Applications,  A. Bermudez de Castro, D. Gomez, P. Quintela, P. Salgado (Eds.), Springer-Verlag Heidelberg, Pages 1192-1200},
	2006.
	
	
	\bibitem{raduAWR}
	F. A. Radu, N. Suciu, J. Hoffmann, A. Vogel, O. Kolditz, C. H. Park, S. Attinger,
	\emph{Accuracy of numerical simulations of contaminant transport in heterogeneous aquifers: a comparative study, Adv. Water Resour., Volume 34, Issue 1, Pages 47-61},
	2011. 
	
	%\bibitem{fem}
	%J. N. Reddy,
	%\emph{An Introduction to the finite Element Method, McGraw-Hill Mechanical Engineering},
	%2006.  
	
	\bibitem{russell1983}
	T. F. Russell, M. F. Wheeler,
	\emph{Finite element and finite difference methods for continuous flows in porous media, SIAM, Pages 35-106},
	1983.  
	
	\bibitem{Slodicka}
	M. Slodicka,
	\emph{A robust and efficient linearization scheme for doubly non-linear and degenerate parabolic problems arising in flow
		in porous media, SIAM J. Sci. Comput., Volume 23, Issue 5, Pages 1593-1614},
	2002.
	
	\bibitem{smith1994}
	J. E. Smith, R. W. Gillham,
	\emph{The effect of concentration-dependent surface tension on the flow of water and transport of dissolved organic compounds: A pressure head-based formulation and numerical model, Water Resour. Res., Volume 31, Issue 3, Pages 343-354},
	1994. 
	
%	\bibitem{surfactant1999}
%	J. E. Smith, R. W. Gillham,
%	\emph{Effects of solute concentration-dependent surface tension on
%		unsaturated flow: Laboratory sand column experiments, Water Resour. Res., Volume 35, Issue 4, Pages 973–982},
%	1999.
	
	\bibitem{laboratory}
	J. Smith, R. Gillham,
	\emph{Effects of solute concentration-dependent surface tension on
		unsaturated flow: Laboratory sand column experiments, Water Resource Research Volume 35, Issue 4, Pages 973-982},
	1999.
	
	\bibitem{suciu2014}
	N. Suciu,
	\emph{Diffusion in random velocity fields with applications to contaminant transport in groundwater, Water Resour. Res., Volume 69, Pages 114-133},
	2014.  
	
	\bibitem{vohralik}
	M. Vohralik,
	\emph{A posteriori error estimates for lowest-order mixed finite element discretizations of convection-diffusion-reaction equations, SIAM J. Numer. Anal. Volume Volume 45, Issue 4, Pages 1570-1599},
	2007.   
	
   \bibitem{wang2013}
   	X. Wang, H. A. Tchelepi,
   \emph{Trust-region based solver for nonlinear transport in heterogeneous porous media, J Comput Phys, Volume 253, Pages 114-137}, 2013.
	
	\bibitem{woodward}
	C. S. Woodward, C. N. Dawson,
	\emph{Analysis of expanded mixed finite element methods for a nonlinear parabolic equation modeling flow into variably saturated porous media, SIAM J. Numer. Anal. Volume 37, Issue 3, Pages 701-724},
	2000.   
	
	 \bibitem{Pop98}
	W. A. Yong, I. S. Pop, \emph{A numerical approach to porous medium equations, Preprint 95-50 (SFB 359), IWR. University of Heidelberg}, 1996. 
	
	\bibitem{younis2010}
	R. Younis, H. A. Tchelepi, K. Aziz, 
	\emph{Adaptively Localized Continuation-Newton Method--Nonlinear Solvers That Converge All the Time, SPE Journal, Volume 15, Issue 02, Pages 526-544}, 2010.
	
	
\end{thebibliography}


\end{document}